\documentclass{article}

\usepackage[utf8]{inputenc}
\usepackage{amsmath,amssymb,amsthm,enumitem,thm-restate,tikz,graphv1}
\usetikzlibrary{decorations.pathreplacing}

\usepackage{hyperref}
\hypersetup
{
    colorlinks=true,
    linkcolor = blue,
    citecolor = red,
}

\usepackage{chngcntr}

\usepackage[ruled,vlined]{algorithm2e}

\usepackage{mathtools}

\title{Variants of the Gy\'arf\'as-Sumner Conjecture:\\Oriented Trees and Rainbow Paths}

\author
{
  \normalsize Manu Basavaraju\thanks{\scriptsize Department of Computer Science and Engineering, National Institute of Technology Karnataka, Surathkal, India. E-mail: \texttt{manub@nitk.ac.in}}
  \and
  \normalsize L. Sunil Chandran\thanks{\scriptsize Department of Computer Science and Automation, Indian Institute of Science, Bangalore, India.
  E-mail: \texttt{sunil@iisc.ac.in}}
  \and
  \normalsize Mathew C. Francis\thanks{\scriptsize Computer Science Unit, Indian Statistical Institute, Chennai, India. E-mail: \texttt{mathew@isichennai.res.in}}
  \and
  \normalsize Karthik Murali\thanks{\scriptsize School of Computer Science, Carleton University, Ottawa, Canada. E-mail: \texttt{karthikmurali@cmail.carleton.ca}}
}

\theoremstyle{plain}
\newtheorem{theorem}{Theorem}

\theoremstyle{plain}
\newtheorem{lemma}[theorem]{Lemma}

\theoremstyle{plain}

\theoremstyle{definition}
\newtheorem{proposition}{Proposition}

\theoremstyle{plain}
\newtheorem{corollary}{Corollary}

\theoremstyle{remark}
\newtheorem{comment}{Remark}

\theoremstyle{plain}

\theoremstyle{plain}
\newtheorem{conjecture}{Conjecture}

\theoremstyle{plain}

\theoremstyle{definition}
\newtheorem{definition}{Definition}

\newcommand{\com}[1]{}
\newcommand{\st}{\ensuremath{\mathrm{\mathsf{st}}}}

\date{\vspace{-5ex}}

\begin{document}

\maketitle

\begin{abstract}
Given a finite family $\mathcal{F}$ of graphs, we say that a graph $G$ is ``$\mathcal{F}$-free'' if $G$ does not contain any graph in $\mathcal{F}$ as a subgraph. We abbreviate $\mathcal{F}$-free to just ``$F$-free'' when $\mathcal{F}=\{F\}$.
A vertex-colored graph \(H\) is called ``rainbow'' if no two vertices of \(H\) have the same color. Given an integer \(s\) and a finite family of graphs \(\mathcal{F}\), let $\ell(s,\mathcal{F})$ denote the smallest integer such that any properly vertex-colored $\mathcal{F}$-free graph $G$ having $\chi(G)\geq \ell(s,\mathcal{F})$ contains an induced rainbow path on $s$ vertices. Scott and Seymour showed that \(\ell(s,K)\) exists for every complete graph $K$. A conjecture of N.~R.~Aravind states that $\ell(s,C_3)=s$. The upper bound on $\ell(s,C_3)$ that can be obtained using the methods of Scott and Seymour setting $K=C_3$ are, however, super-exponential. Gy\'arf\'as and S\'ark\"ozy showed that \(\ell(s,\{C_3,C_4\}) = \mathcal{O}\big((2s)^{2s}\big)\). For $r\geq 2$, we show that \(\ell(s,K_{2,r}) \leq (r-1)(s-1)(s-2)/2+s\) and therefore, \(\ell(s,C_4) \leq \frac{s^2-s+2}{2}\). This significantly improves Gy\'arf\'as and S\'ark\"ozy's bound and also covers a bigger class of graphs. We adapt our proof to achieve much stronger upper bounds for graphs of higher girth: we prove that \(\ell(s,\{ C_3,C_4,\ldots,C_{g-1}\})\leq s^{1+\frac{4}{g-4}}\), where $g\geq 5$. Moreover, in each case, our results imply the existence of at least \(s!/2\) distinct induced rainbow paths on \(s\) vertices. Along the way, we obtain some new results on an oriented variant of the Gy\'arf\'as-Sumner conjecture. For $r\geq 2$, let $\mathcal{B}_r$ denote the orientations of $K_{2,r}$ in which one vertex has out-degree or in-degree $r$. We show that every $\mathcal{B}_r$-free oriented graph having chromatic number at least $(r-1) (s-1) (s-2) + 2s +1$ and every bikernel-perfect oriented graph with girth $g\geq 5$ having chromatic number at least $2s^{1+\frac{4}{g-4}}$ contains every oriented tree on at most $s$ vertices as an induced subgraph.
\bigskip

\end {abstract}

\section{Introduction}\label{sec:introduction}
All the graphs considered in this paper are simple and finite. They are also considered to be undirected unless otherwise mentioned. Please refer to~\cite{Diestel} for standard graph theoretic terminology and notation that is not defined in this paper. Given a (directed or undirected) graph \(G=(V,E)\), its vertex set and edge set are denoted by \(V(G)\) and \(E(G)\), respectively. Given a directed graph $G$, we let $\hat{G}$ denote the underlying undirected graph of $G$: more precisely, $V(\hat{G})=V(G)$ and $E(\hat{G})=\{uv\colon (u,v)\in E(G)$ or $(v,u)\in E(G)\}$. If $G$ is an undirected graph, then we define $\hat{G}=G$. If \(\mathcal{F}\) is a finite set of (directed) graphs, a (directed) graph \(G\) is said to be \(\mathcal{F}\)-free if there is no subgraph in \(G\) which is isomorphic to a member of \(\mathcal{F}\). If \(\mathcal{F}\) is a singleton with the element \(F\), we shall simply write \(F\)-free instead of \(\{F\}\)-free.
If $\mathcal{F}$ is a finite family of undirected graphs, a directed graph $G$ is said to be $\mathcal{F}$-free if $\hat{G}$ is $\mathcal{F}$-free.

All the colorings considered in this paper are proper vertex colorings---i.e. they are assignments of colors to the vertices such that no two adjacent vertices get the same color. The sets of colors that we use shall always be subsets of natural numbers. The minimum number of colors required to properly color a graph \(G\) is called its \emph{chromatic number}, denoted by \(\chi(G)\). A proper coloring which uses exactly \(\chi(G)\) colors is called an \emph{optimal coloring}. Given a properly colored graph \(G\), a subgraph \(H\) of \(G\) is said to be \emph{rainbow} if no two vertices of \(H\) have the same color. Additionally, if the subgraph \(H\) is induced, then \(H\) is said to be an \emph{induced rainbow} subgraph of \(G\). A \(t\)-path is a path on \(t\) vertices. A cycle of length \(n\) is denoted by \(C_n\) (the length of a path or cycle is the number of edges in it). A complete graph on $n$ vertices is denoted by $K_n$.

If $G$ is a directed graph, then a path in $G$ simply refers to a path in $\hat{G}$ and a cycle in $G$ simply refers to a cycle in $\hat{G}$. The \emph{girth} of a (directed or undirected) graph \(G\), denoted by \(g(G)\), is the minimum length of a cycle in $G$. A \emph{directed path} in a directed graph $G$ is a sequence of pairwise distinct vertices $v_1,v_2,\ldots,v_t$ such that for $1\leq i<t$, $(v_i,v_{i+1})\in E(G)$.  A \emph{directed cycle} in a directed graph $G$ is a sequence of vertices $v_1,v_2,\ldots,v_t,v_1$ such that $v_1,v_2,\ldots,v_t$ is a directed path in $G$ and $(v_t,v_1)\in E(G)$. An \emph{orientation} of an undirected graph \(G\) is a directed graph obtained by assigning a direction to each edge in \(E(G)\). An \emph{acyclic orientation} of a graph is one in which there are no directed cycles. A directed graph $G$ is said to be an \emph{oriented graph} if it is an orientation of $\hat{G}$; in other words, it is a directed graph with no directed cycle of length 2.
An \emph{oriented tree} is an orientation of a tree. Let \(T\) be an oriented tree with a vertex \(r \in V(T)\) designated as the root. If for each vertex \(v \in V(T)\), there is a directed path starting from \(r\) and ending at \(v\), then \(T\) is called an \emph{out-tree}. On the other hand, if for each vertex \(v \in V(T)\), there is a directed path starting from \(v\) and ending at \(r\), then \(T\) is called an \emph{in-tree}. A directed graph without any directed cycles is called an \emph{directed acyclic graph}. There is a linear ordering of the vertices of every directed acyclic graph, called a \emph{topological ordering} of the graph, such that for every edge $(u,v)$ in the graph, the vertex $u$ comes before the vertex $v$ in the ordering. A \emph{kernel} of a directed graph $G$ is an independent set $S$ of $G$ having the additional property that for every vertex $u\in V(G)\setminus S$, there exists a vertex $v\in S$ such that $(u,v)\in E(G)$. Similarly, an \emph{antikernel} of $G$ is an independent set $S\subseteq V(G)$ having the property that for every vertex $u\in V(G)\setminus S$, there exists a vertex $v\in S$ such that $(v,u)\in E(G)$ (note that some authors use the terms ``kernel'' and ``antikernel'' with their meanings interchanged). A directed graph is called \emph{bikernel-perfect} if every induced subgraph of it has both a kernel and an antikernel. It is well known that directed acyclic graphs are all bikernel-perfect~\cite{richardson}.
\section{Preliminaries}
Erd\H{o}s~\cite{ErdosProbability} proved that given any integer \(k\), there exist graphs with girth more than \(k\) and arbitrarily large chromatic number. This is equivalent to the following statement.

\begin{theorem}[Erd\H{o}s~\cite{ErdosProbability}]\label{thm:erdos}
If $\mathcal{F}$ is a finite family of graphs, none of which is a forest, then for every integer $c$, there exists an $\mathcal{F}$-free graph $G$ having $\chi(G)\geq c$.
\end{theorem}

It is folklore that every graph $G$ contains every tree (and hence every forest) on $\chi(G)$ vertices as a subgraph. In fact, every graph having minimum degree $k-1$ contains every tree on $k$ vertices as a subgraph. 
Now combining this with Theorem~\ref{thm:erdos}, we get the following.

\begin{proposition}\label{prop:boundedchi}
Let $\mathcal{F}$ be a finite family of graphs. There exists a constant $c$ such that $\chi(G)<c$ for every $\mathcal{F}$-free graph $G$ if and only if $\mathcal{F}$ contains a forest.
\end{proposition}

\subsection{The Gy\'arf\'as-Sumner Conjecture}\label{sec:gyarfassumner}

The \emph{clique number} $\omega(G)$ of a graph $G$ is the maximum size of a clique in $G$.
For any graph \(G\), a natural lower bound for its chromatic number \(\chi(G)\) is its clique number \(\omega(G)\). There are several constructions in the literature that show that there exist triangle-free graphs having arbitrarily large chromatic number (for example, see Mycielski's construction in \cite{DouglasWest}), which means that
\(\chi(G)\) can be arbitrarily large compared to \(\omega(G)\). Thus, it is interesting to study those classes of graphs for which \(\chi(G)\) can be bounded above by a function of \(\omega(G)\). Formally, a class of graphs is said to be \(\chi\)-\textit{bounded} if there exists a function \(f\) such that for every graph \(G\) in the class, \(\chi(G) < f(\omega(G))\). We phrase this in slightly different language: for a family $\mathcal{G}$ of graphs and a complete graph $K$, let $f(\mathcal{G},K)$ be the smallest integer such that every $K$-free graph $G$ in $\mathcal{G}$ has $\chi(G)<f(\mathcal{G},K)$. Thus, a class of graphs $\mathcal{G}$ is $\chi$-bounded if and only if $f(\mathcal{G},K)$ exists for every complete graph $K$. Clearly, for a positive integer $t$, $f(\mathcal{G},K_t)$ exists only if $f(\mathcal{G},K_i)$ exists for each $i<t$. From the existence of constructions like that of Mycielski's, we know that the class of all graphs, in fact, even the class of triangle-free graphs, is not \(\chi\)-bounded---i.e., if we denote the class of all graphs by $\mathcal{G}$, then even $f(\mathcal{G},K_3)$ does not exist.

We denote by $\mathcal{G}_A$ the class of all graphs that do not contain the graph $A$ as an induced subgraph. For brevity, we abuse notation and let $f(A,K)$ denote $f(\mathcal{G}_A,K)$. Now we can think of $f(A,K)$ in the following way: $f(A,K)$ is the smallest integer such that every $K$-free graph $G$ having $\chi(G)\geq f(A,K)$ contains $A$ as an induced subgraph, or put in other words, every graph that has chromatic number at least $f(A,K)$ contains either $A$ as an induced subgraph or $K$ as a subgraph.
Since the class of all graphs is not $\chi$-bounded, it is natural to ask the following question: Does there exist a graph \(A\) such that $\mathcal{G}_A$ is \(\chi\)-bounded? In other words, is there a graph $A$ such that $f(A,K)$ exists for every complete graph $K$? Clearly, there are such graphs---for example, a $K_2$---and they are called \emph{\(\chi\)-bounding graphs}. Even though $K_2$ is \(\chi\)-bounding, Mycielski's construction shows that \(K_3\) is not. In fact, it follows from Theorem~\ref{thm:erdos} that no \(\chi\)-bounding graph can have a cycle. Thus, the only possible \(\chi\)-bounding graphs are forests.
Gy\'arf\'as~\cite{GyarfasRamsey} and Sumner~\cite{SumnerSubtrees} independently conjectured that the converse must also be true, i.e. all forests are \(\chi\)-bounding. In fact, as noted in~\cite{Gyarfas1987Problems}, since every forest is an induced subgraph of some tree, this is equivalent to saying that all trees are $\chi$-bounding.

\begin{conjecture}[Gy\'arf\'as-Sumner]
For every tree $T$ and complete graph $K$, $f(T,K)$ exists. That is, there exists an integer $f(T,K)$ such that all \(K\)-free graphs \(G\) with \(\chi(G) \geq f(T,K)\) contain \(T\) as an induced subgraph.
\end{conjecture}

Note that if the above conjecture is true, then for any positive integer $s$ and complete graph $K$, there exists a number $f(s,K)$ such that every $K$-free graph $G$ having $\chi(G)\geq f(s,K)$ contains every tree on $s$ vertices as an induced subgraph---in fact, $f(s,K)=\max\{f(T,K)\colon T$ is a tree on $s$ vertices$\}$. Despite considerable work, this conjecture is known to be true only for simple types of trees (see \cite{ScottSeymourSurvey} for a complete list of trees for which the conjecture has been solved). It is not even known whether the conjecture holds for triangle-free graphs; i.e. whether or not $f(s,K_3)$ exists for every integer $s$.

Generalizing the above notation, for a finite family $\mathcal{F}$ of graphs and a graph $A$, we can define $f(A,\mathcal{F})$ to be the least integer such that every $\mathcal{F}$-free graph $G$ having $\chi(G)\geq f(A,\mathcal{F})$ contains $A$ as an induced subgraph. By Proposition~\ref{prop:boundedchi}, we have that $f(A,\mathcal{F})$ exists whenever $\mathcal{F}$ contains a forest, irrespective of the graph $A$. Since our aim is to study those graphs $A$ that appear as induced subgraphs in $\mathcal{F}$-free graphs when their chromatic number becomes sufficiently large, we shall limit our study to those families $\mathcal{F}$ that contain no forest. Then from Theorem~\ref{thm:erdos}, it follows that $f(A,\mathcal{F})$ exists only if $A$ is a forest. As before, it can be seen that $f(A,\mathcal{F})$ exists for all forests $A$ if and only if $f(T,\mathcal{F})$ exists for all trees $T$. Following our earlier notation, we define $f(s,\mathcal{F})=\max\{f(T,\mathcal{F})\colon T$ is a tree on $s$ vertices$\}$.

Some versions of the Gy\'arf\'as-Sumner conjecture obtained by relaxing the condition of $K$-freeness to $\mathcal{F}$-freeness, where $\mathcal{F}$ is some finite family of undirected graphs, are known to be true. In other words, $f(s,\mathcal{F})$ exists for some families $\mathcal{F}$ of graphs. For example, as the following theorem shows, even though it is unknown whether $f(s,C_3)$ exists (note that $C_3=K_3$), we know that $f(s,\{C_3,C_4\})$ exists.

\begin{theorem}[Gy\'arf\'as, Szemeredi, Tuza~\cite{GyarfasSzemerediTuza}]\label{thm:gst}
Every $\{C_3,C_4\}$-free graph $G$ with $\chi(G)\geq s$ contains every tree on $s$ vertices as an induced subgraph. In other words, $f(s,\{C_3,C_4\})\leq s$.
\end{theorem}

The detailed survey of Scott and Seymour~\cite{ScottSeymourSurvey} lists many known results on the Gy\'arfas-Sumner conjecture, and discusses two well-studied variants of the conjecture: the ``oriented'' variant and the ``rainbow'' variant. We mainly concern ourselves with these two variants, which as we shall see, are closely related. Below, we give a brief introduction to each of these variants and present our contributions.

\subsection{Our results}\label{sec:ourresults}
The Gallai-Roy-Vitaver theorem states that every properly colored graph $G$ contains a (not necessarily induced) rainbow $\chi(G)$-path. It is natural to ask if a stronger statement, along the lines of the Gy\'arf\'as-Sumner conjecture, holds: for every integer $s\geq 1$, does there exist an integer $h(s)$ such that every properly colored graph $G$ having $\chi(G)\geq h(s)$ contains every tree on $s$ vertices as a rainbow subgraph? This is not true as shown by Erd\H{o}s and Hajnal~\cite{ErdosHajnal}, who showed the existence of graphs with arbitrarily large chromatic number that can be properly colored in such a way that the neighbourhood of every vertex is colored using just two colors (the so-called ``shift graph of triples''; see~\cite{KiersteadAndTrotter,ScottSeymourSurvey}). Thus, there are properly colored graphs with arbitrarily large chromatic number that do not even contain a rainbow claw. Note that the graphs produced by the construction of Erd\H{os} and Hajnal are also triangle-free. N.~R.~Aravind (see~\cite{Manu}) proposed to modify the Gallai-Roy-Vitaver theorem in another way by asking the following question: does every properly colored triangle-free graph $G$ contain an induced rainbow $\chi(G)$-path? Note that such a statement cannot be true for general graphs, as for example, the complete graph does not contain even an induced 3-path in spite of having arbitrarily large chromatic number. On the other hand, it is known that every triangle-free graph $G$ contains an induced $\chi(G)$-path~\cite{Gyarfas1987Problems}. We state Aravind's conjecture below.

\begin{conjecture}[N. R. Aravind]\label{conj:aravind}
Any properly colored triangle-free graph \(G\) contains an induced rainbow path on \(\chi(G)\) vertices.
\end{conjecture}

This conjecture can be generalized along the lines of the Gy\'arf\'as-Sumner conjecture as follows. Let $\mathcal{F}$ be a finite family of graphs. We define \(\ell(s,\mathcal{F})\) to be the smallest integer such that all properly colored \(\mathcal{F}\)-free graphs \(G\) with \(\chi(G) \geq \ell(s,\mathcal{F})\) contain an induced rainbow $s$-path. Clearly, $\ell(s,\emptyset)$ does not exist for $s\geq 3$ (since, as noted above, there exist graphs having arbitrarily large chromatic number and containing no induced 3-paths). Scott and Seymour~\cite{ScottAndSeymour} showed that $\ell(s,K)$ exists for every complete graph $K$, or in other words, that all rainbow paths are $\chi$-bounding (here we mean that the class of all properly colored graphs that do not contain an induced rainbow $s$-path, for some integer $s\geq 0$, is $\chi$-bounded). But the upper bounds provided by their construction are very fast growing functions. For example, their bound for \(\ell(s,C_3)\) is super-exponential in \(s\), whereas Conjecture~\ref{conj:aravind} states that \(\ell(s,C_3) = s\).

Note that Conjecture~\ref{conj:aravind} holds trivially for the case when \(g(G) > \chi(G)\) because the rainbow \(\chi(G)\)-path in $G$ (which exists by the Gallai-Roy-Vitaver theorem) 
is also an induced path in this case. Babu et al.~\cite{Manu} showed Conjecture~\ref{conj:aravind} to be true for the case when \(g(G) = \chi(G)\). Therefore, as a special case, Conjecture~\ref{conj:aravind} holds true for 4-chromatic triangle-free graphs. Gy\'arf\'as and S\'ark\"ozy~\cite{GyarfasAndSarkozy} showed that $\ell(s,\{C_3,C_4\})$ exists. From their proof, it follows that \(\ell(s,\{C_3,C_4\}) = \mathcal{O}\big((2s)^{2s}\big)\). 
In this paper, we improve this result significantly by bringing this bound down to a quadratic function of $s$; we show that $\ell(s,C_4)=\mathcal{O}(s^2)$. We in fact show the more general result that $\ell(s,K_{2,r})=\mathcal{O}(rs^2)$ for any $r\geq 2$ (note that $C_4$ is the same as $K_{2,2}$). The exact result is stated below.
\medskip

\begin{restatable}{theorem}{rainbowpaththm}\label{thm:rainbowpaths}
Let $G$ be a properly colored graph. In each of the following cases, $G$ contains at least $s!/2$ distinct
induced rainbow paths on $s$ vertices:
\begin{enumerate}
\vspace{-0.05in}
\renewcommand{\labelenumi}{(\roman{enumi})}
\renewcommand{\theenumi}{(\roman{enumi})}
\itemsep 0.01in
\item\label{rainbowpathk2r} \(G\) is \(K_{2,r}\)-free (\(r \geq 2\)) and \(\chi(G)\geq (r-1)(s-1)(s-2)/2 + s\).
\item\label{rainbowpathgirth} \(G\) has girth $g\geq 5$ and \(\chi(G) \geq s^{1+\frac{4}{g-4}}\).
\end{enumerate}
\end{restatable}

\begin{comment}
As the theorem above implies that $\ell(s,C_4)\leq \frac {s^2 - s+2} {2} $, it improves the bound given by Gy\'arf\'as and S\'ark\"ozy and also covers a bigger class of graphs. For graphs of higher girth, our results imply that \(\ell(s,\mathcal{C}_g) \leq s^{1+\frac{4}{g-4}}\) where \(\mathcal{C}_{g} = \{C_3,C_4,\dots,C_{g-1}\}\) and \(g \geq 5\). Note that the method of Gy\'arf\'as and S\'ark\"ozy does not yield a bound which improves as the girth of the graph increases.
\end{comment}
\medskip

\begin{comment}
Although Babu et al.~\cite{Manu} proved that \(\ell(s,\mathcal{C}_s) =s\), it seems extremely difficult by their method to even show that \(\ell(s,\mathcal{C}_{s-1})\) is linear in $s$. However, our result above (Theorem~\ref{thm:rainbowpaths}\ref{rainbowpathgirth}) shows that linear bounds on \(\ell(s,\mathcal{C}_g)\) can be achieved even when \(g = \Omega(\log s)\). Note that this is equivalent to saying that
there exists a function $h$ such that for every $\epsilon<1$, every properly colored graph $G$ having girth at least $h(\epsilon)\log\chi(G)$ contains an induced rainbow path on at least $\epsilon\cdot\chi(G)$ vertices \big(for example, we could let $h(\epsilon)=\frac{4}{\log(1/\epsilon)}$~\big).
\end{comment}
\medskip

\begin{comment}
Another feature of our method, in comparison to the previous approaches, is that we are able to show the presence of at least \(s!/2\) distinct induced rainbow \(s\)-paths as opposed to a single rainbow $s$-path.
\end{comment}
\bigskip

Let $G$ be an undirected graph, $\alpha$ be a proper vertex coloring of $G$, and $<$ a total order defined on the colors used in the coloring $\alpha$.
Assign a direction to each edge $uv$ of $G$ as follows: If $\alpha(u) <  \alpha(v)$, direct the edge from $v$ to $u$; otherwise,
direct it from $u$ to $v$. Let the resulting oriented graph be $D$. It is easy to verify that $D$ is a directed acyclic graph, since along any directed path, the colors of the vertices are in decreasing order. We refer to $D$ as the \emph{natural orientation} of $G$ with respect to $(\alpha, <)$. Whenever the ordering $<$ is clear from the context, we may drop $<$ and just call it ``the natural orientation of $G$ with respect to $\alpha$''. Note that any induced directed path of $D$ would correspond to an induced rainbow path in $G$ and any induced rainbow path in $G$ along which the colors decrease corresponds to an induced directed path in $D$.
\medskip

Let $\mathcal{F}$ be a finite collection of directed or undirected graphs. Let $\overline{\ell}(s,\mathcal{F})$ denote the least integer such that every $\mathcal{F}$-free directed acyclic graph $G$ with $\chi(G)\geq\overline{\ell}(s,\mathcal{F})$ contains an induced directed $s$-path.
Note that a properly colored graph contains an induced rainbow path on $s$ vertices if its natural orientation $D$ (with respect to an arbitrarily chosen ordering of the colors) contains an induced directed path on $s$ vertices. It follows that $\ell(s,\hat{\mathcal{F}})\leq\overline{\ell}(s,\mathcal{F})$, where $\hat{\mathcal{F}}=\{\hat{G}\colon G\in\mathcal{F}\}$.
This means that in order to prove Theorem~\ref{thm:rainbowpaths}\ref{rainbowpathk2r}, it suffices to show that $\overline{\ell}(s,K_{2,r}) \le (r-1)(s-1)(s-2)/2 + s$ and in order to prove Theorem~\ref{thm:rainbowpaths}\ref{rainbowpathgirth}, it suffices to show that for $g\geq 5$, $\overline{\ell}(s,\{C_3,C_4,\ldots,C_{g-1}\}) \le s^{1+\frac{4}{g-4}}$. In fact, we prove a stronger statement: we show that if the chromatic number of a suitably restricted directed acyclic graph is at least the bounds mentioned above, then it contains every out-tree and every in-tree on $s$ vertices as an induced subgraph. Also, for the former case, we prove a slightly stronger statement, namely instead of $K_{2,r}$-free directed acylic graphs, we prove the bound for the more general class of directed acyclic graphs that do not contain as a subraph an orientation of a $K_{2,r}$ having a vertex of out-degree or in-degree $r$. Formally, for $r\geq 2$, let $\mathcal{B}^+_r$ denote the family of orientations of $K_{2,r}$ in which at least one vertex has out-degree $r$, and $\mathcal{B}^-_r$ denote the family of orientations of $K_{2,r}$ in which at least one vertex has in-degree $r$. We also define $\mathcal{B}_r=\mathcal{B}^+_r\cup\mathcal{B}^-_r$.

\begin{restatable}{theorem}{outtreethm}\label{thm:outtree}
Let $G$ be a directed acyclic graph. Then in each of the following cases, $G$ contains every out-tree (resp. every in-tree) on at most $s$ vertices as an induced subgraph:
\begin{enumerate}
\vspace{-0.05in}
\renewcommand{\labelenumi}{(\roman{enumi})}
\renewcommand{\theenumi}{(\roman{enumi})}
\itemsep 0.01in
\item\label{outtbr} $G$ is $\mathcal{B}^+_r$-free (resp. $\mathcal{B}^-_r$-free), where $r\geq 2$, and $\chi(G)\geq (r-1)(s-1)(s-2)/2+s$.
\item\label{outtgirth} $G$ has girth $g\geq 5$ and \(\chi(G) \geq s^{1+\frac{4}{g-4}}\).
\end{enumerate}
\end{restatable}

Could it be possible that a theorem such as the one above could be made for arbitrary oriented trees instead of just out-trees and in-trees, and also without restricting ourselves to directed acyclic graphs? In fact, if we only want the oriented trees to occur as subgraphs that are not necessarily induced, then such a result exists due to Burr~\cite{Burr}.

\begin{theorem}[Burr~\cite{Burr}]\label{thm:burr}
Every directed graph $G$ having $\chi(G)\geq (s-1)^2$ contains every oriented forest on $s$ vertices as a subgraph.
\end{theorem}

Note that the above theorem does not require any restrictions on the graph $G$. But what if we want to show that every oriented tree on $s$ vertices appears as an \emph{induced} subgraph of $G$ when $\chi(G)$ is large enough? Clearly, we will need some restrictions on $G$ in this case, as orientations of complete graphs cannot contain induced subgraphs isomorphic to any oriented tree on more than two vertices. It turns out that if we subject $G$ to restrictions similar to those in Theorem~\ref{thm:outtree}, then every oriented tree can be shown to be an induced subgraph of $G$, even without requiring $G$ to be acyclic, provided that $\chi(G)$ is large enough. 
In particular, we generalize Theorem~\ref{thm:outtree}\ref{outtbr} by showing the existence of \emph{every} oriented tree on $s$ vertices as an induced subgraph in \emph{any} $\mathcal{B}_r$-free oriented graph whose chromatic number is just a little over two times the bound in Theorem~\ref{thm:outtree}\ref{outtbr}.

\begin{restatable}{theorem}{orientedgstthm}\label{thm:orientedgyarfas}
	Let $G$ be a $\mathcal{B}_r$-free oriented graph ($r\geq 2$) with $\chi(G) \geq (r-1)(s-1)(s-2)+2s+1$. Then $G$ contains every oriented tree on $s$ vertices as an induced subgraph.
\end{restatable}

We then generalize Theorem~\ref{thm:outtree}\ref{outtgirth} by showing the existence of \emph{every} oriented tree on $s$ vertices as an induced subgraph in \emph{any} bikernel-perfect oriented graph (directed acyclic graphs form a subclass of bikernel-perfect oriented graphs) having girth at least 5 whose chromatic number is at least two times the bound in Theorem~\ref{thm:outtree}\ref{outtgirth}.

\begin{restatable}{theorem}{bikernelthm}\label{thm:bikernel}
	Let $G$ be a bikernel-perfect oriented graph having girth $g\geq 5$. Then $G$ contains every oriented tree on $s$ vertices as an induced subgraph if $\chi(G)\geq 2s^{1+\frac{4}{g-4}}$.
\end{restatable}

A relevant question at this point would be whether the natural generalization of the Gy\'arf\'as-Sumner conjecture to oriented graphs could also be true: i.e. could it be true that for every integer $s$, every oriented graph of high enough chromatic number that does not contain a fixed oriented complete graph $K$ as a subgraph contains every oriented tree on $s$ vertices as an induced subgraph? As the following discussion shows, this is not the case.
Given a finite collection of directed or undirected graphs ${\cal F}$ and an oriented graph $A$, let $\overline{f}(A,{\cal F})$ be the smallest integer such that every ${\cal F}$-free oriented graph $G$ having $\chi(G)\geq \overline{f}(A,{\cal F})$ contains $A$ as an induced subgraph. Since a graph contains $\hat{A}$ as an induced subgraph if some orientation of it contains $A$ as an induced subgraph, we have $f(\hat{A},\hat{\mathcal{F}})\leq\overline{f}(A,\mathcal{F})$. Note that if $\mathcal{F}$ contains an oriented forest, then by Theorem~\ref{thm:burr}, we have that $\overline{f}(A,\mathcal{F})$ exists for every oriented graph $A$. So the families $\mathcal{F}$ that we consider do not contain any oriented forest. Then as noted before, $f(\hat{A},\hat{\mathcal{F}})$ exists only when $\hat{A}$ is a forest, which implies by the previous inequality that $\overline{f}(A,\mathcal{F})$ exists only if $A$ is an oriented forest. Thus, as before, we will focus on the case when $A$ is an oriented tree. As with undirected trees, we say that an oriented tree $T$ is $\chi$-bounding if $\overline{f}(T,K)$ exists for all complete graphs $K$.
Clearly, if $\overline{f}(T,K)$ exists for every oriented tree $T$ and any complete graph $K$, then the Gy\'arf\'as-Sumner conjecture would also hold. However, it is known that even \(\overline{f}(T,C_3)\) does not exist when $T$ corresponds to certain orientations of a 4-path: Kierstead and Trotter showed that \(\rightarrow\rightarrow\rightarrow\) is not \(\chi\)-bounding~\cite{KiersteadAndTrotter} and Gy\'arf\'as showed that \(\rightarrow\leftarrow\rightarrow\) is not \(\chi\)-bounding~\cite{GyarfasProblem115}. Therefore, $\overline{f}(T,K)$ can exist for every complete graph $K$ only for certain restricted kinds of oriented trees. Recently, Chudnovsky, Scott and Seymour showed that the family of oriented stars and the remaining orientations of a 4-path are all \(\chi\)-bounding~\cite{chudnovsky}.

Further questions in this direction could involve finding out whether specific orientations of larger trees are \(\chi\)-bounding. However, proceeding in this direction may be tedious. A more fruitful approach would be to alter  the constraint of triangle-freeness and ask whether $\overline{f}(T,C_4)$ exists for every oriented tree $T$. In other words, is it true that every oriented tree appears as an induced subgraph of any \(C_4\)-free oriented graph with a high enough chromatic number?
We show that this is indeed the case: in particular, Theorem~\ref{thm:orientedgyarfas} shows that $\overline{f}(T,\mathcal{B}_r)$ exists for every oriented tree $T$ and $r\geq 2$. As every $C_4$-free oriented graph is also $\mathcal{B}_2$-free, this result is an oriented variant of Gy\'arf\'as, Szemeredi and Tuza's result that $f(T,\{C_3,C_4\})$ exists for every undirected tree $T$ (Theorem~\ref{thm:gst}). 
As before, for a positive integer $s$ and a finite collection of directed or undirected graphs ${\cal F}$, we let $\overline{f}(s,{\cal F})=\max\{\overline{f}(T,{\cal F})\colon T$ is an oriented tree on $s$ vertices$\}$ --- i.e., it is the minimum integer such that every ${\cal F}$-free oriented graph having chromatic number at least $\overline{f}(s,{\cal F})$ contains every oriented tree on $s$ vertices as an induced subgraph. Clearly, $f(s,\hat{\cal F})\leq \overline{f}(s,{\cal F})$. Theorem~\ref{thm:orientedgyarfas} implies that $\overline{f}(s,\mathcal{B}_r)\leq (r-1)(s-1)(s-2)+2s+1$. Thus, $\overline{f}(s,C_4)\leq s^2-s+3$.

\subsection{Notation}
We define some common notation that we shall use in our proofs.
Let $G$ be an undirected or directed graph.
An undirected edge between the vertices $u$ and $v$ of $G$ shall be denoted as $uv$, whereas an edge directed from $u$ to $v$ shall be denoted as $(u,v)$. Two vertices are said to be \emph{adjacent} in $G$ if there is a (directed or undirected) edge between them in $G$. The set of all vertices adjacent to a vertex $v$ in $G$ is called the \emph{neighbourhood} of \(v\) in $G$ and is denoted by \(N_G(v)\) (the subscript $G$ is omitted when the graph under consideration is clear). In a directed graph $G$, if $(u,v)\in E(G)$, then $v$ is said to be an \emph{out-neighbour} of $u$, and $u$ is said to be an \emph{in-neighbour} of $v$. For a directed graph $G$ and a vertex $v\in V(G)$, the \emph{out-neighbourhood} (resp. \emph{in-neighbourhood}) of $v$, i.e. the set of out-neighbours (resp. in-neighbours) of $v$ in $G$, is denoted by $N^+_G(v)$ (resp. $N^-_G(v)$). (Again, the subscript is omitted when the graph under consideration is clear from the context.) Note that for every vertex $v\in V(G)$, $N(v)=N^+(v)\cup N^-(v)$.
Also \emph {out-degree}  (resp. \emph {in-degree}) of a vertex $v$ is defined as $|N^+(v)|$ (resp. $|N^-(v)|$) and is denoted as $d^+(v)$ (resp. $d^-(v)$), with sometimes a subscript denoting the graph under consideration.  For a graph or digraph $G$ and a vertex $v\in V(G)$, we denote by $N_G[v]$ the \emph{closed neighbourhood} of $v$, i.e. $N_G[v]=N_G(v)\cup\{v\}$.
Given a coloring $\alpha$ of a (directed or undirected) graph $G$ and $S\subseteq V(G)$, we denote by $\alpha(S)$ the set $\{\alpha(v)\colon v\in S\}$.
For a graph (resp. digraph) and a set $S\subseteq V(G)$, we denote by $G[S]$ the subgraph (resp. subdigraph) of $G$ induced by the vertices $S$. We denote by $G-S$ the graph $G[V(G)\setminus S]$.
\section{Induced rainbow paths}
We first prove Lemma~\ref{lem:generallemma}, which will be our main tool for later results. In order to prove the lemma, we need the following Tur\'an-type results from extremal graph theory.

\medskip
\noindent Given a finite family of graphs \(\mathcal{L}\), let \(\text{\textbf{ex}}(n,\mathcal{L})\) denote the maximum number of edges that an \(\mathcal{L}\)-free graph \(G\) on \(n\) vertices can have.

\begin{lemma}[see Theorem~4.1 and Theorem~4.4 of~\cite{FurediAndSimonovits}] \label{lem:extremal}\leavevmode
	\begin{enumerate}
	\vspace{-0.05in}
	\renewcommand{\labelenumi}{(\roman{enumi})}
	\renewcommand{\theenumi}{(\roman{enumi})}
	\itemsep 0.01in
		\item \label{ext1} \(\textnormal{\textbf{ex}}(n,\{C_3,C_4,\dots,C_{2k}\}) < \frac{1}{2}n^{1+\frac{1}{k}} + \frac{1}{2}n.\)
		\item \label{ext2} \(\textnormal{\textbf{ex}}(n,\{C_3,C_4,\dots,C_{2k+1}\}) < \big(\frac{n}{2}\big)^{1+\frac{1}{k}} + \frac{1}{2}n.\)
	\end{enumerate}
\end{lemma}

\begin {definition}[Out-tree coloring]
Let $\alpha$ be a proper vertex coloring of an oriented graph $G$ and let $<$ be a total ordering of $\alpha(V(G))$. We say that $(\alpha,<)$ is an \emph{out-tree coloring} of $G$ if for every $v\in V(G)$, $\alpha(N^+(v))\supseteq\{i\in\alpha(V(G))\colon i<\alpha(v)\}$.  When the ordering $<$ on the colors is clear from the context, we shorten ``$(\alpha,<)$ is an out-tree coloring'' to just ``$\alpha$ is an out-tree coloring''.
\end {definition}

Recall that the girth of an oriented graph is simply the girth of its underlying undirected graph (see Section~\ref{sec:introduction}).

\begin{lemma}\label{lem:generallemma}
Let $F$ be an oriented graph and let $G$ be a subgraph of $F$. Suppose that $G$ has an out-tree coloring using at least $k$ colors. In each of the following cases, every out-tree tree $T$ on $s$ vertices is present as an induced rainbow subgraph of $G$, which is also an induced subgraph of $F$:
\begin{enumerate}
	\vspace{-0.05in}
	\renewcommand{\labelenumi}{(\roman{enumi})}
	\renewcommand{\theenumi}{(\roman{enumi})}
	\itemsep 0.01in
	\item \label{lembr} \(F\) is $\mathcal{B}^+_r$-free (\(r \geq 2\)) and \(k\geq (r-1)(s-1)(s-2)/2 + s\).
    \item \label{lemgirth} $F$ has girth $g\geq 5$ and \(k \geq s^{1+\frac{4}{g-4}}\).
\end{enumerate}
\end{lemma}
\begin{proof}
We assume that the given out-tree coloring $(\alpha,<)$ of $G$ uses natural numbers as colors, where the ordering $<$ between the colors is the usual ordering $<$ on natural numbers. For each vertex $v\in V(G)$ and $i\in\{1,2,\ldots,\alpha(v)-1\}$, we denote by $v^i$ an arbitrarily chosen vertex in $N_G^+(v)$ such that $\alpha(v^i)=i$ (such a vertex exists as $\alpha$ is an out-tree coloring of $G$).
Let \(T\) be any out-tree on \(s\) vertices and let \(w_1,w_2,\ldots,w_s\) be a topological ordering of \(T\). For each $i\in\{2,3,\ldots,s\}$, let $p(i)$ be the integer such that $w_{p(i)}$ is the parent of $w_i$ in $T$. Clearly $w_1$ corresponds to the root of $T$, and since the parent of every vertex appears before it in the ordering, we have $p(i)<i$.

Let $i\in\{1,2,\ldots,s\}$. Define $T_i=T[\{w_1,w_2,\ldots,w_i\}]$. Suppose that $T'_i$ is an induced tree in $G$ as well as $F$ that is isomorphic to $T_i$. Let $v_1,v_2,\ldots,v_i$ be the vertices of $T'_i$ corresponding to $w_1,w_2,\ldots,w_i$. Then clearly, for each $j\in\{2,3,\ldots,i\}$, $v_{p(j)}$ is the parent of $v_j$ in $T'_i$. We say that $T'_i$ is a ``good'' tree if $\alpha(v_1)=k$ and for $2\leq j\leq i$, 
\begin{enumerate}
\item $\alpha(v_j)<\alpha(v_{j-1})$, and
\item for $\alpha(v_j)<t<\alpha(v_{j-1})$, the subgraph induced by $\{v_1,v_2,\ldots,v_{j-1},v_{p(j)}^t\}$ in $F$ is not isomorphic to $T_j$. This means that in $F$, the vertex $v_{p(j)}^t$ has a neighbour other than $v_{p(j)}$ in $\{v_1,v_2,\ldots,v_{j-1}\}$, i.e. $N_F(v_{p(j)}^t)\cap (\{v_1,v_2,\ldots,v_{j-1}\}\setminus\{v_{p(j)}\})\neq\emptyset$.  (Refer Figure~\ref{fig:scheme} for an example that shows how a good tree might look like.)
\end{enumerate}

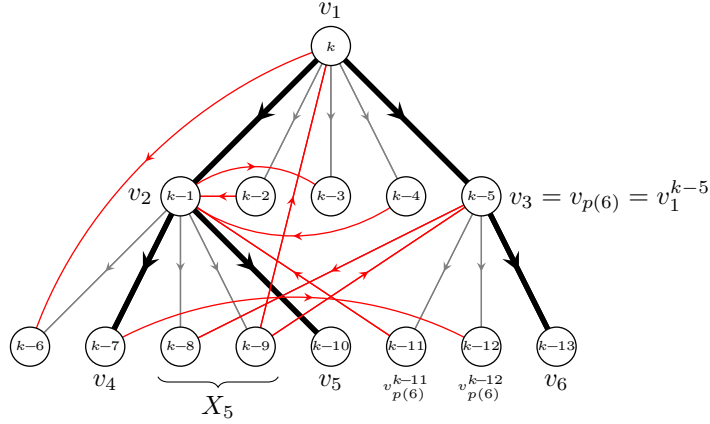
\begin{figure}
\begin{center}
\begin{tikzpicture}[scale=2]
\newcommand{\myptr}{{\arrow{stealth}}}
\renewcommand{\defradius}{0.13}
\renewcommand{\isdirected}{myptr}
\renewcommand{\vertexset}{(v1,5,5),(v2,4,4),(x1,4.5,4),(x2,5,4),(x3,5.5,4),(v3,6,4),(v4,3.5,3),(v5,5,3),(v6,6.5,3),(x4,3,3),(x5,4,3),(x6,4.5,3),(x7,5.5,3),(x8,6,3)}
\renewcommand{\edgeset}{(v1,v2,,2),(v1,v3,,2),(v1,x1,gray),(v1,x2,gray),(v1,x3,gray),(x1,v2,red),(v2,x2,red,,15),(x3,v2,red,,20),(v2,v4,,2),(v2,v5,,2),(v3,v6,,2),(v2,x4,gray),(v1,x4,red,,-25),(v2,x5,gray),(v2,x6,gray),(v3,x5,red),(x6,v1,red),(x6,v3,red),(v3,x7,gray),(x7,v2,red),(v3,x8,gray),(v4,x8,red,,25,,,0.55)}
\drawgraph
\node[above=7] at (5,5) {$v_1$};
\node at (5,5) {\tiny $k$};
\node[left=7] at (4,4) {$v_2$};
\node at (4,4) {\tiny $k\!\!-\!\!1$};
\node at (4.5,4) {\tiny $k\!\!-\!\!2$};
\node at (5,4) {\tiny $k\!\!-\!\!3$};
\node at (5.5,4) {\tiny $k\!\!-\!\!4$};
\node[right=7] at (6,4) {$v_3=v_{p(6)}=v_1^{k-5}$};
\node at (6,4) {\tiny $k\!\!-\!\!5$};
\node at (3,3) {\tiny $k\!\!-\!\!6$};
\node[below=7] at (3.5,3) {$v_4$};
\node at (3.5,3) {\tiny $k\!\!-\!\!7$};
\node at (4,3) {\tiny $k\!\!-\!\!8$};
\node at (4.5,3) {\tiny $k\!\!-\!\!9$};
\node[below=7] at (5,3) {$v_5$};
\node at (5,3) {\tiny $k\!\!-\!\!10$};
\node at (5.5,3) {\tiny $k\!\!-\!\!11$};
\node[below=7] at (5.5,3) {\tiny $v_{p(6)}^{k\!-\!11}$};
\node at (6,3) {\tiny $k\!\!-\!\!12$};
\node[below=7] at (6,3) {\tiny $v_{p(6)}^{k\!-\!12}$};
\node[below=7] at (6.5,3) {$v_6$};
\node at (6.5,3) {\tiny $k\!\!-\!\!13$};
\draw [decorate,decoration={brace,amplitude=4pt}] (4.65,2.8) -- (3.85,2.8);
\node at (4.25,2.6) {$X_5$};
\end{tikzpicture}
\end{center}
\caption{A possible way the good tree $T'_6$ might look like. The edges of $T'_6$ are drawn in bold. For each vertex $u$, its color $\alpha(u)$ is written inside it. For each $j\in\{2,3,\ldots,6\}$ and $\alpha(v_j)<t<\alpha(v_{j-1})$, a gray arc shows the edge from $v_{p(j)}$ to $v_{p(j)}^t$. A red edge shows an adjacency in $F$ between $v_{p(j)}^t$ and some vertex in $\{v_1,v_2,\ldots,v_{j-1}\}\setminus\{v_{p(j)}\}$.}\label{fig:scheme}
\end{figure}

Notice that every good tree is an induced rainbow subgraph of $G$.
In order to prove the statement of the lemma, we shall prove the stronger statement that there is a good tree isomorphic to $T_s=T$.
Suppose for the sake of contradiction that there is no good tree isomorphic to $T_s$. Let $i$ be the largest integer in $\{1,2,\ldots,s-1\}$ such that there is a good tree $T'_i$ isomorphic to $T_i$ (recall that there is a vertex $y\in V(G)$ such that $\alpha(y)=k$; this means that the single vertex $y$ is a good tree $T'_1$ isomorphic to $T_1$, and so $i$ exists).
As before, let $v_1,v_2,\ldots,v_i$ be the vertices of $T'_i$ corresponding to the vertices $w_1,w_2,\ldots,w_i$ of $T_i$. For $2\leq j\leq i$, define $X_j=\left\{v_{p(j)}^{\alpha(v_j)+1},v_{p(j)}^{\alpha(v_j)+2},\ldots,v_{p(j)}^{\alpha(v_{j-1})-1}\right\}$ and $X'=\left\{v_{p(i+1)}^1,v_{p(i+1)}^2,\ldots,v_{p(i+1)}^{\alpha(v_i)-1}\right\}$. It can be seen that these sets are well defined as follows. Let $j\in\{2,3,\ldots,i\}$. Recall that $p(j)<j$, which means that $p(j)\leq j-1$. Since $T'_i$ is a good tree, we then get that $\alpha(v_{p(j)})\geq\alpha(v_{j-1})>\alpha(v_j)$. Then since $\alpha$ is an out-tree coloring, we can conclude that $v_{p(j)}$ has at least one out-neighbour having each color less than $\alpha(v_{j-1})$. Thus, each of the vertices $v_{p(j)}^{\alpha(v_{j-1})-1}, v_{p(j)}^{\alpha(v_{j-1})-2},\ldots,v_{p(j)}^{\alpha(v_j)+2},v_{p(j)}^{\alpha(v_j)+1}$ exist. Therefore, the set $X_j$ is well defined for each $j\in\{2,3,\ldots,i\}$ (one such set is marked in Figure~\ref{fig:scheme} as an illustration). It can be seen in a similar way that the set $X'$ is also well defined.

We now define $X=X'\cup\bigcup_{j=2}^i X_j$. (It is not difficult to see that we always have $X_2=\emptyset$, although we do not use this fact.) Note that for each color in $\{1,2,\ldots,k\}\setminus\{\alpha(v_1),\alpha(v_2),\ldots,\alpha(v_i)\}$, the set $X$ contains exactly one vertex having that color. Therefore, we have 

\begin{equation}\label{eq1}
    k = i + |X|
\end{equation}

As $T'_i$ is a good tree, we know that for every vertex $x\in X_j$, where $2\leq j\leq i$, $N_F(x)\cap(\{v_1,v_2,\ldots,v_{j-1}\}\setminus\{v_{p(j)}\})\neq\emptyset$.
If for some integer $t\in\{1,2,\ldots,\alpha(v_i)-1\}$, we have $N_F(v_{p(i+1)}^t)\cap(\{v_1,v_2,\ldots,v_i\}\setminus\{v_{p(i+1)}\})=\emptyset$, then taking $t$ to be the largest such integer, we have that $G[\{v_1,v_2,\ldots,v_i\}\cup\{v_{p(i+1)}^t\}]$ is a good tree isomorphic to $T_{i+1}$, which contradicts our choice of $i$. We can therefore conclude that for every $x\in X'$, $N_F(x)\cap(\{v_1,v_2,\ldots,v_i\}\setminus\{v_{p(i+1)}\})\neq\emptyset$.
\medskip

\noindent \textbf{Proof of~\ref{lembr}}
\medskip

\noindent
Let $j\in\{2,3,\ldots,i\}$. Since \(F\) is \(\mathcal{B}^+_r\)-free (\(r \geq 2\)), any vertex $z\in\{v_1,v_2,\ldots,v_{j-1}\}\setminus\{v_{p(j)}\}$ can have at most $r-1$ vertices in $X_j$ as neighbours in $F$, as otherwise there is a subgraph in $F$ that is isomorphic to a graph in $\mathcal{B}^+_r$,  whose vertex set is $\{v_{p(j)},z\}$ together with $r$ vertices in $N_F(z)\cap X_j$.
As every vertex in $X_j$ has some neighbour in $\{v_1,v_2,\ldots,v_{j-1}\}\setminus\{v_{p(j)}\}$, we can conclude that $|X_j|\leq (r-1)(j-2)$. Similarly, no vertex in $\{v_1,v_2,\ldots,v_i\}\setminus\{v_{p(i+1)}\}$ can have $r$ neighbours in $X'$ (in $F$), implying that $|X'|\leq (r-1)(i-1)$. We thus get:

\[|X| \leq (r-1)(i-1)+\sum_{j=2}^i (r-1)(j-2) = (r-1)(i)(i-1)/2\] 

\noindent By plugging this into Equation~(\ref{eq1}),
we get:
\[k \: \leq \: i + \frac{(r-1)(i)(i-1)}{2} \]

\noindent
As $i\leq s-1$, this gives:
$$k\leq \frac {(r-1)(s-1)(s-2)} {2}  + s -1$$
We now have a contradiction to the assumption that $k\geq\frac {(r-1)(s-1)(s-2)} {2}  + s$ and this completes the proof.
\medskip

\noindent \textbf{Proof of~\ref{lemgirth}}
\medskip

\noindent
Recall that we had defined a ``path'' or ``cycle'' in an oriented graph to be a path or cycle, respectively, in its underlying undirected graph (see Section~\ref{sec:introduction}).
Let us first show that $i\geq g$. Suppose that $i<g$.
If there exists $x\in X_j$ for some $j\in\{2,3,\ldots,i\}$, then we know that there exists $z\in N_F(x)\cap\{v_1,v_2,\ldots,v_{j-1}\}\setminus\{v_{p(j)}\}$. But then the path in $T'_i$ between $z$ and $v_{p(j)}$ (this path contains only vertices in $\{v_1,v_2,\ldots,v_{j-1}\}$) together with $x$ forms a cycle of length at most $j\leq i<g$ in $F$. Since this contradicts the fact that $g$ is the girth of $F$, we can conclude that $X_j=\emptyset$ for $2\leq j\leq i$. Further, if there exists $x\in X'$, then since it has as neighbours the vertex $v_{p(i+1)}$ and also a vertex $z\in N_F(x)\cap\{v_1,v_2,\ldots,v_i\}\setminus\{v_{p(i+1)}\}$, we can conclude that the path in $T'_i$ between $v_{p(i+1)}$ and $z$ contains at least $g-2\geq i-1$ edges. Since $T'_i$ contains only $i$ vertices, this means that $\hat{T'_i}$ is the path $v_1,v_2,\ldots,v_i$ and that $N_F(x)\cap\{v_1,v_2,\ldots,v_i\}=\{v_1,v_i\}$. Further, if there are two vertices $x,y\in X'$, this means that $x,v_1,y,v_i,x$ is a cycle of length 4 in $F$, which is a contradiction to the fact that $g\geq 5$. This implies that $|X'|\leq 1$. Then $|X|\leq 1$, which by Equation~(\ref{eq1}) means that $i\geq k-1$. From the statement of the lemma, we have $k\geq s+1$ as $k$ and $s$ are integers, implying that \(i\geq s\), contradicting the assumption that $i<s$. We can therefore assume that $i\geq g$.

We shall construct an auxiliary undirected graph \(H\) with \(V(H) = V(T'_i)\) and \(E(H) = E(\hat{T'_i}) \cup E'\), where \(E'\) is a set of edges defined as follows. For $j\in\{2,3,\ldots,i\}$ and vertex \(x \in X_j\), the edge \(v_{p(j)}y\) is added to \(E'\), where $y$ is an arbitrarily chosen vertex in $N_F(x)\cap (\{v_1,v_2,\ldots,v_{j-1}\}\setminus\{v_{p(j)}\})$. Note that if $v_{p(j)}$ and $y$ are adjacent in $F$, then $v_{p(j)},x,y$ form a triangle in $F$, contradicting the fact that $g(F)=g\geq 5$. Also, no other neighbour of $v_{p(j)}$ can be adjacent to $y$ in $F$, since otherwise there will be a cycle of length four in $F$, again contradicting the fact that $g(F)=g\geq 5$. These two observations together imply that we never add an edge of $E(\hat{F})$ to $E'$ and also that we never add the same edge twice to $E'$. Similarly, for each $x\in X'$, the edge $v_{p(i+1)}y$ is added to $E'$, where $y$ is an arbitrarily chosen vertex in $N_F(x)\cap (\{v_1,v_2,\ldots,v_i\}\setminus\{v_{p(i+1)}\})$. Again, arguing as before, we can see that we never add the same edge twice to $E'$ and that $E'$ and $E(\hat{F})$ are disjoint. We can therefore conclude that $H$ is a simple graph and that \(|E'| = |X'|+\sum_{j=2}^i |X_j|= |X|\).

Now, we claim that \(g(H) \geq g/2\). Given any cycle \(C\) of \(H\), each edge of \(C\) can be replaced by at most 2 edges of \(F\) (which in effect is like subdividing that edge) to obtain a closed walk (in fact, a cycle) in $F$ containing at most $2|E(C)|$ edges. This can be seen as follows. Let $C=u_0,u_1,\ldots,u_{l-1},u_0$ be a cycle in $H$ and let the edges of $E'$ that belong to $C$ be $u_{j_1}u_{j_1+1},u_{j_2}u_{j_2+2},\ldots,u_{j_q}u_{j_q+1}$ where $j_1<j_2<\cdots<j_q$ and subscripts are modulo $l$. Recalling the definition of $E'$, we know that for each $t\in\{1,2,\ldots,q\}$, there exists a vertex $x_t\in X_2\cup X_3\cup\cdots\cup X_i\cup X'$ such that $u_{j_t}x_t,u_{j_t+1}x_t\in E(F)$. Then $C'=u_0,u_1,\ldots,u_{j_1},x_t,u_{j_1+1},\ldots,u_{j_2},x_2,u_{j_2+1},\ldots,u_{j_q},x_q,u_{j_q+1},\ldots,u_{l-1},u_0$ is a closed walk in $F$ that contains at most $2l$ edges. Thus, if \(g(H) < g/2\), then there exists a cycle in \(F\) of length less than \(g\), which is a contradiction.
\medskip

\noindent From Lemma~\hyperref[ext1]{\ref{lem:extremal}\ref{ext1}}, if \(g(H)\) is odd, we get (note that if $g(H)=3$, then the inequality below holds trivially): 
\[|E(H)| \  < \  \frac{i^{1+2/(g(H)-1)} + i}{2} \  \leq \  \frac{i^{1+4/(g-2)} + i}{2} \ < \ \frac{i^{1+4/(g-4)} + i}{2}\]
\medskip

\noindent From Lemma~\hyperref[ext2]{\ref{lem:extremal}\ref{ext2}}, if \(g(H)\) is even, we get: 
\[|E(H)| \  < \  \bigg(\frac{i}{2}\bigg)^{1+2/(g(H)-2)} + \frac{i}{2} \  \leq \  \bigg(\frac{i}{2}\bigg)^{1+4/(g-4)} + \frac{i}{2}  \ < \ \frac{i^{1+4/(g-4)} + i}{2}\]
\medskip

\noindent Combining both the cases, we get 
\[|E(H)| \ < \ \frac{i^{1+4/(g-4)} + i}{2}\]

\noindent Since \(|E(H)| = |E'|+(i-1)\), we have:
\[|E'| < \frac{i^{1+4/(g-4)} - i + 2}{2}\]

\noindent Since \(|E'| = |X|\), by plugging the above inequality into Equation~(\ref{eq1}),
we get:
\begin{eqnarray*}
k \: &<& \: i + \frac{i^{1+4/(g-4)} - i + 2}{2}=\frac{i(1 + 2/i+i^{4/(g-4)})}{2}\\
&\leq&\frac{i(e^{2/i}+i^{4/(g-4)})}{2}\qquad\mbox{(since }1+2/i\leq e^{2/i}\mbox{)}\\
&\leq&\frac{i(e^{2/(g-4)}+i^{4/(g-4)})}{2}\qquad\mbox{(since $i\geq g\geq g-4$)}\\
&=&\frac{i(\sqrt{e}^{4/(g-4)}+i^{4/(g-4)})}{2}
\end{eqnarray*}
By our assumption that $i\geq g$, we get that \(i \geq 2 > \sqrt{e}\). Using this in the above inequality, we get:
\[k < i^{1+4/(g-4)}\]
Now by using the bound on $k$ given in the statement of the lemma, we get:
$$ s^{1+4/(g-4)} \leq k < i^{1+4/(g-4)}$$
This implies that $i\geq s$, which is a contradiction.
\end{proof}

The following corollary simply states the case when $F=G$ in Lemma~\ref{lem:generallemma}.
\begin{corollary}\label{cor:general}
Let $G$ be an oriented graph that has an out-tree coloring using at least $k$ colors. Then in each of the following cases, every out-tree on $s$ vertices is present as an induced rainbow subgraph of $G$:
\begin{enumerate}
\vspace{-0.05in}
\renewcommand{\labelenumi}{(\roman{enumi})}
\renewcommand{\theenumi}{(\roman{enumi})}
\itemsep 0.01in
\item \label{corbr} \(G\) is $\mathcal{B}^+_r$-free (\(r \geq 2\)) and \(k\geq (r-1)(s-1)(s-2)/2 + s\).
\item \label{corgirth} $G$ has girth $g\geq 5$ and \(k \geq s^{1+\frac{4}{g-4}}\).
\end{enumerate}
\end{corollary}
\bigskip

Recall that our intention is to prove Theorem~\ref{thm:rainbowpaths}. Given a graph $G$ with a proper vertex coloring $\beta$ and a total ordering $<$ on the set of
colors, a convenient orientation of $G$ that we can consider is the
natural orientation $D$ of $G$ with respect to $(\beta, <)$, as we had discussed in Section~\ref{sec:ourresults}. If $(\beta,<)$ were an \emph{out-tree coloring} of $D$, we could have immediately
inferred that there is an induced rainbow path on $s$ vertices in $G$ by applying Corollary~\ref{cor:general} on $D$, setting $T$ to be a directed path on $s$ vertices. 
But note that it is possible that $(\beta,<)$ is not an out-tree coloring of $D$ for any choice of the total ordering $<$ of the colors of $\beta$.

To overcome this difficulty, we recolor $G$ using the following Greedy Refinement Algorithm developed and used in~\cite{Manu} (this algorithm is adapted from a folklore recoloring procedure). The algorithm takes as input a graph with a proper coloring $\beta$ where some total order on  the set of colors is also given. The algorithm gives a new proper coloring $\alpha$ as output. 

\medskip
\begin{algorithm}[H]
	\label{GreedyRefinement}
	\SetAlgoLined
	\SetKwInOut{Input}{Input}
	\SetKwInOut{Output}{Output}
	\Input{A graph \(G=(V,E)\) with a proper coloring \(\beta\) and an ordering $<$ of the colors used by $\beta$.}
	\Output{The graph \(G=(V,E)\) with a refined coloring \(\alpha\)}
	\BlankLine
	Let $\beta_1,\beta_2,\ldots,\beta_h$ be the colors used by $\beta$, where $\beta_1<\beta_2< \cdots<\beta_h$.\medskip\\
	\ForEach{\(v \in V(G)\)}
	{\(\alpha(v) \gets 0\)} 
	\For{\(i \gets 1\) \KwTo \(h\)}
	{
		\ForEach{\(v\in V(G)\) \textnormal{such that} \(\beta(v)=\beta_i\) \textnormal{and} \(\alpha(v)=0\)}
		{
			\(\alpha(v)\leftarrow\min (\mathbb{N}^{>0} \setminus \{\alpha(u)\colon u\in N(v)\})\) 
			
			(i.e., \(\alpha(v)\) is set equal to the least positive integer that has not already been assigned to a neighbour of \(v\)).
		}
	}
	\caption{Greedy Refinement Algorithm}
\end{algorithm}
\bigskip


\begin{proposition}[\cite{Manu}]\label{prop1}
For every vertex \(v\), there exists a subset \(S=\{u_1,u_2, \dots, u_{j-1}\}\) of \(N(v)\), where $j=\alpha(v)$, such that for \(1 \leq i \leq j-1\), \(\alpha(u_i)=i\) and \(\beta(u_i) < \beta(v)\).
\end{proposition}

\begin{proof}
Let \(v\) be any vertex with \(\alpha(v)=j\). From Algorithm \ref{GreedyRefinement}, it is clear that at the point in time when \(v\) was colored with \(j\), there existed a subset \(S=\{u_1,u_2,\ldots, u_{j-1}\}\) of \(N(v)\) such that \(\alpha(u_i)=i\) for each $i\in\{1,2,\ldots,j-1\}$. Because each such \(u_i \in S\) should have been colored before \(v\), we have \(\beta(u_i) \leq\beta(v)\), and since $u_i\in N(v)$, we further have $\beta(u_i)<\beta(v)$. As Algorithm \ref{GreedyRefinement} colors each vertex only once, the subset \(S\) of \(N(v)\) exists even when the algorithm terminates.
\end{proof}

The following proposition is easy to see.

\newcommand{\refinement}{\ensuremath{\mathrm{refinement}}}
\begin{proposition}\label{prop:optimal}
The coloring $\alpha$ generated by the Greedy Refinement Algorithm (Algorithm~\ref {GreedyRefinement}) always uses at most as many colors as that used by the input coloring $\beta$.
\end{proposition}

The coloring $\alpha$ generated by Algorithm~\ref{GreedyRefinement} will be referred to as $\refinement(\beta,<)$, the ``refinement'' of the proper coloring $\beta$ with respect to the total
order $<$ on the set of colors. Note that the colors used by $\alpha$ are from $\mathbb{N}^{>0}$ and unless otherwise specified, the color set of $\alpha$ is ordered by the usual order on natural numbers. 
Note that $\alpha =\refinement(\beta,<)$ is an out-tree coloring of $D$, the natural orientation of $G$ with respect to $\alpha$. This follows from Proposition~\ref{prop1} and the definition of
natural orientation (note that the vertices $u_1,u_2, \ldots, u_{\alpha(v)-1}$ that are guaranteed to exist for every vertex $v$ in Proposition~\ref{prop1} are all out-neighbours of $v$ in $D$).

\begin{definition}[Decreasing tree]\label{def:dectree}
Given an undirected graph $G$ with a proper vertex coloring $\beta$ and an ordering of the colors of $\beta$, a tree  $T$ in $G$ with a chosen root is said to be a \emph{decreasing tree} in $G$ if the colors on every path from root to leaf in $T$ are in decreasing order. If in addition, $T$ is a path and one of the endpoints of this path is the root of $T$, then $T$ is said to be a \emph{decreasing path} in $G$. Clearly, a decreasing path in $G$ is also a rainbow path in $G$.
\end{definition}

\begin{lemma}\label{lem:decreasing_tree_lemma}
Let $G$ be an undirected graph with a proper coloring $\beta$ on whose colors a total order $<$ is defined. Then in each of the following cases, every rooted tree on $s$ vertices is present in $G$ as an induced decreasing tree:
\begin{enumerate}
\vspace{-0.05in}
\renewcommand{\labelenumi}{(\roman{enumi})}
\renewcommand{\theenumi}{(\roman{enumi})}
\itemsep 0.01in
\item \(G\) is \(K_{2,r}\)-free (\(r \geq 2\)) and \(\chi(G)\geq (r-1)(s-1)(s-2)/2 + s\).
\item $G$ has girth \(g \geq 5\) and \(\chi(G) \geq s^{1+4/(g-4)}\).
\end{enumerate}
\end {lemma}
\begin{proof}
Let $T$ be a rooted tree on $s$ vertices. Assign directions to the edges of $T$ so that we get an out-tree $T'$ with the same root as $T$. Let $\alpha=\refinement(\beta,<)$ be a coloring that uses the colors $\{1,2,\ldots,k\}$. Clearly, $k\geq \chi(G)$. Let $D$ be the natural orientation of $G$ with respect to $\alpha$. Let $D'$ be the subgraph of $D$ that has vertex set $V(D')=V(D)$ and edge set $E(D')=\{(u,v)\in E(D)\colon\beta(u)>\beta(v)\}$. By Proposition~\ref{prop1}, we know that $\alpha$ is an out-tree coloring of $D'$ as well. Since $V(D')=V(D)$, we know that there is at least one vertex $x$ in $D'$ having $\alpha(x)=k$. 
It can be seen that if one of the conditions in the statement of the lemma is satisfied, then Lemma~\ref{lem:generallemma} can be applied to $D$, $D'$ and $T'$ in order to conclude there is an induced subgraph of $D'$, which is also induced in $D$, that is isomorphic to $T'$. Since the colors in $\beta$ decrease along any directed path in $D'$, we have that the induced subgraph of $D$ that is isomorphic to $T'$ is a decreasing tree with respect to $(\beta,<)$.
\end{proof}
\medskip

Theorem~\ref{thm:rainbowpaths} follows from Lemma~\ref{lem:decreasing_tree_lemma} in a straightforward way. 

\rainbowpaththm*

\begin{proof}
Let $k$ be the number of colors used in the proper coloring of $G$. For a given ordering of the colors in $G$, we can apply Lemma~\ref{lem:decreasing_tree_lemma} on $G$ and a path on $s$ vertices (choosing any endpoint of the path as its root), and thereby conclude that there is a decreasing path on $s$ vertices that is induced in $G$. As noted in Definition~\ref{def:dectree}, this path is clearly also a rainbow path in $G$, and therefore it is an induced rainbow path in $G$.
Note that a particular decreasing path $P$ can be produced by the application of Lemma~\ref{lem:decreasing_tree_lemma} for at most $2k!/s!$ different orderings of the $k$ colors of the proper coloring of $G$. Thus, if we apply Lemma~\ref{lem:decreasing_tree_lemma} for all the $k!$ different orderings possible for the colors in $G$, we get at least $s!/2$ different induced rainbow paths on $s$ vertices in $G$.
\end{proof}

Theorem~\hyperref[rainbowpathk2r]{\ref{thm:rainbowpaths}\ref{rainbowpathk2r}} shows that every properly colored $C_4$-free graph $G$ having $\chi(G)\geq\frac{s^2-s+2}{2}$ contains $s!/2$ induced rainbow paths on $s$ vertices. This improves the result of Gy\'arf\'as and Sark\"ozy, who showed that there exists a function $h(s)=\Omega((2s)^{2s})$ such that every graph having girth at least 5 and chromatic number at least $h(s)$ contains an induced rainbow path on $s$ vertices, not only in the bound, but also in the fact that it guarantees the existence of several rainbow paths. Notice that Theorem~\hyperref [rainbowpathk2r]{\ref {thm:rainbowpaths}\ref{rainbowpathk2r}} applies to a wider class of graphs than the graphs with girth at least 5 and that Theorem~\hyperref[rainbowpathgirth]{\ref{thm:rainbowpaths}\ref{rainbowpathgirth}} gives better and better bounds as the girth increases.

If \(c\) is a proper coloring of a graph \(G\), let \(\mu(G,c)\) be the number of vertices in the largest induced rainbow path in \(G\) with respect to \(c\). Let \(\mu(G) := \min_c\mu(G,c)\) over all proper colorings \(c\) of \(G\).
In the context of induced rainbow paths, Theorem~\hyperref[rainbowpathgirth]{\ref{thm:rainbowpaths}\ref{rainbowpathgirth}} shows that linear bounds on $\ell(s,\mathcal{C}_g)$ can be achieved when \(g = \Omega(\log s)\). More precisely, if \(G\) is a graph with girth \(g \geq 4 + 4\log s\) and \(\chi(G) \geq 2s\), then \(\mu(G) \geq s\). Babu et al.~\cite{Manu} proved that \(\ell(s,\mathcal{C}_s) =s\), i.e. $\mu(G) = \chi(G)$ if $\chi(G) \ge g(G)$. They asked whether there exists some constant $c<1$ for which it can be shown that $\mu(G)=\chi(G)$ for every graph $G$ having girth at least $c\cdot\chi(G)$. Note that our result above implies that
there exists a function $h$ such that for every $\epsilon<1$, $\mu(G)\geq\epsilon\cdot\chi(G)$ for every graph $G$ having girth at least $h(\epsilon)\log\chi(G)$.

\subsection{Induced rainbow trees in optimal colorings}\label{sec:rainbowtrees}

As noted in Section~\ref{sec:ourresults}, a construction of Erd\H{o}s and Hajnal~\cite{ErdosHajnal} shows that Conjecture~\ref{conj:aravind} cannot be extended from rainbow paths to rainbow trees.
Indeed, Erd\H{o}s et al.~\cite{Erdosetal} (Theorem~2.8) show that there exist graphs of arbitrarily large chromatic number and arbitrarily large girth that can be properly colored in such a way that the neighbourhood of each vertex contains at most 2 colors. Thus, even though every tree $T$ is an induced subgraph of a graph $G$ having girth at least 5 and sufficiently high chromatic number (Theorem~\ref{thm:gst}), the existence of $T$ as a rainbow subgraph cannot be guaranteed in a properly colored graph $G$ no matter how high the chromatic number or girth of $G$ are. Thus, the ``rainbow version'' of the Gy\'arf\'as-Sumner conjecture is not true. But note that even though the graphs constructed in~\cite{ErdosHajnal} and~\cite{Erdosetal} have proper colorings in which the neighbourhood of each vertex contains at most two colors (and hence one cannot find even a rainbow claw as a subgraph, given this coloring), these colorings are not optimal. Is there a way to color these graphs optimally while still ensuring that no rainbow claw appears as a subgraph? The answer to this question is negative. In fact,
Gy\'arf\'as, Szemeredi and Tuza~\cite{GyarfasSzemerediTuza} prove the following result, while stating that it is a part of graph theoretic folklore.

\begin{theorem}[Gy\'arf\'as-Szemeredi-Tuza~\cite{GyarfasSzemerediTuza}]\label{thm:rainbowsubgraph}
Let $G$ be an optimally colored graph. Then every tree on $\chi(G)$ vertices occurs as a rainbow subgraph of $G$.
\end{theorem}

Does every tree on $s$ vertices occur as an \emph {induced} rainbow subgraph of any \emph{optimally} colored $K$-free graph with high enough chromatic number (where $K$ is some complete graph)? 
Clearly, an affirmative answer to this question would imply the Gy\'arf\'as-Sumner conjecture.
We now show that it can be inferred  from Theorem~\ref{thm:rainbowsubgraph} that an optimally colored $C_4$-free graph $G$ having $\chi(G)\geq (2s)^s$ contains every tree on $s$ vertices as an induced rainbow subgraph. In fact, the following more general statement is true.

\begin{theorem}\label{thm:optimal}
Let $G$ be a $K_{2,r}$-free graph ($r\geq 2$) such that $\chi(G)\geq (rs)^s$. Then for every optimal coloring of $G$, every tree on $s$ vertices is present as an induced rainbow subgraph of $G$.
\end{theorem}
\begin{proof}
Consider an optimal coloring of $G$.
Let $T$ be a complete $(rs)$-ary tree having $s$ levels (a rooted tree in which all internal vertices have exactly $rs$ children and every leaf is at a distance of $s-1$ from the root). It follows from Theorem~\ref{thm:rainbowsubgraph} that $T$ occurs as a rainbow subgraph of $G$. It can be seen in the following way that any tree $H$ having at most $s$ vertices occurs as a subgraph of $T$ that is also an induced subgraph of $G$. We prove this by induction on $|V(H)|$. If $|V(H)|\leq 1$, then the statement is trivially true. Suppose that $|V(H)|\geq 2$. Let $u$ be a leaf of $H$ and $v$ its unique neighbour in $H$. Further, let $H'=H-u$. Let $T'$ be the subtree induced in $T$ by the internal vertices of $T$. Note that $T'$ contains as a subgraph a complete $(r(s-1))$-ary tree having $s-1$ levels. Therefore, since $|V(H')|\leq s-1$, by the inductive hypothesis, there is a subgraph $H''$ of $T'$ isomorphic to $H'$ that is also an induced subgraph of $G$. Let $v'$ be the vertex corresponding to $v$ in $V(H'')\subseteq V(T)$. Note that since $|N_{H''}(v')|\leq s-2$, we have $|N_T(v')\setminus N_{H''}(v')|\geq rs-(s-2)$. Therefore, we can choose distinct vertices $u_1,u_2,\ldots,u_{rs-s}\in N_T(v')\setminus N_{H''}(v')$. Note that since $H''$ is an induced subgraph of $G$, $\{u_1,u_2,\ldots,u_{rs-s}\}\cap V(H'')=\emptyset$. For every vertex $x\in V(H'')\setminus\{v'\}$, we have $|N_G(x)\cap\{u_1,u_2,\ldots,u_{rs-s}\}|\leq r-1$, as otherwise, $r$ neighbours of $x$ in $\{u_1,u_2,\ldots,u_{rs-s}\}$ together with $x$ and $v'$ will form a $K_{2,r}$ in $G$. Thus $\sum_{x\in V(H'')\setminus\{v'\}} |N_G(x)\cap\{u_1,u_2,\ldots,u_{rs-s}\}|\leq (r-1)(s-2)<rs-s$. Therefore, there exists $i\in\{1,2,\ldots,rs-s\}$ such that $N_G(u_i)\cap (V(H'')\setminus\{v'\})=\emptyset$. Now $G[V(H'')\cup\{u_i\}]$ is an induced subgraph of $G$ that is isomorphic to $H$ and is also a subgraph of $T$. Thus every tree on $s$ vertices occurs as a subgraph of $T$ that is also an induced subgraph of $G$. Since $T$ is rainbow, it follows that every tree on $s$ vertices occurs as an induced rainbow subgraph in $G$.
\end{proof}
\medskip

Next, we show that if we are ready to weaken the statement of Theorem~\ref{thm:optimal} so that we no longer want every optimal coloring of $G$ to have the property stated in the theorem, but only want to guarantee the existence of \emph{some} optimal coloring of $G$ such that every tree on $s$ vertices occurs as an induced rainbow subgraph of $G$, then we can get much better bounds for $\chi(G)$. Notice that instead of restricting the graph to be $K_{2,r}$-free, we could also ask whether any $K$-free graph (where $K$ is some complete graph) with a high enough chromatic number has an optimal coloring such that the graph contains every tree on $s$ vertices as an induced rainbow subgraph. Again, an affirmative answer to this question would imply the Gy\'arf\'as-Sumner conjecture.

\begin{lemma}\label{lem:property_of_refined_coloring}
Let $G$ be an undirected graph with a proper coloring $\beta$. Let $<$ be a total order on the colors of $\beta$ and $\alpha=\refinement(\beta,<)$. Then in each of the following cases, every tree on $s$ vertices is present as an induced rainbow subgraph of $G$ with respect to the coloring $\alpha$:
\begin{enumerate}
\vspace{-0.05in}
\renewcommand{\labelenumi}{(\roman{enumi})}
\renewcommand{\theenumi}{(\roman{enumi})}
\itemsep 0.01in
\item \(G\) is \(K_{2,r}\)-free (\(r \geq 2\)) and \(\chi(G)\geq (r-1)(s-1)(s-2)/2 +s\).
\item $G$ has girth $g\geq 5$ and \(\chi(G) \geq s^{1+4/(g-4)}\).
\end{enumerate}
\end{lemma}

\begin{proof}
Let $T$ be any tree on $s$ vertices. Assign directions to the edges of $T$ so that we get an out-tree $T'$.
Consider the natural orientation $D$ of $G$ with respect to $\alpha$.
By Proposition~\ref{prop1}, it follows that $\alpha$ is an out-tree coloring of $D$.
Let the colors used by $\alpha$ be $\{1,2,\ldots,k\}$. Clearly, $k\geq\chi(G)=\chi(D)$. Notice that if $G,s$ satisfy one of the conditions in the statement of the lemma, then $D,s$ satisfy the corresponding condition in Corollary~\ref{cor:general}.
Thus we can apply Corollary~\ref{cor:general} to conclude that there is an induced rainbow subgraph of $D$ isomorphic to $T'$ with respect to the coloring $\alpha$, which implies that $T$ is an induced rainbow subgraph of $G$ with respect to $\alpha$. 
\end{proof}

\begin{theorem}
Let $G$ be a graph. In each of the following cases, there exists an optimal coloring of $G$ such that every tree $T$ on $s$ vertices occurs as an induced rainbow subgraph of $G$:
\begin{enumerate}
\vspace{-0.05in}
\renewcommand{\labelenumi}{(\roman{enumi})}
\renewcommand{\theenumi}{(\roman{enumi})}
\itemsep 0.01in
\item \(G\) is \(K_{2,r}\)-free (\(r \geq 2\)) and \(\chi(G)\geq (r-1)(s-1)(s-2)/2 + s\).
\item $G$ has girth \(g\geq 5\) and \(\chi(G) \geq s^{1+4/(g-4)}\).
\end{enumerate}
\end{theorem}
\begin{proof}
Let $\beta$ be an arbitrary optimal coloring of $G$, and $<$ be any total order on the set of colors. Let $\alpha=\refinement(\beta, <)$. By Proposition~\ref{prop:optimal}, we have that $\alpha$ is an optimal coloring of $G$. 
Then by Lemma~\ref{lem:property_of_refined_coloring}, every tree on $s$ vertices is present as an induced rainbow subgraph with respect to the optimal coloring $\alpha$ in $G$.
\end{proof}

\section{Induced oriented trees}
\subsection{Out-trees and in-trees in directed acyclic graphs}

\outtreethm*
\begin{proof}
We shall first prove the theorem for the case when $T$ is an out-tree.
Let $\beta: V(G) \rightarrow \mathbb{N}^{>0}$ be a proper vertex  coloring of $\hat G$,  such that $G$ is the natural orientation of $\hat G$ with respect to  $\beta$. 
Such a coloring $\beta$ can be obtained as follows.
Let $G_0=G$. For each $i\geq 1$, define $C_i=\{u\in V(G_{i-1})\colon d^+_{G_{i-1}}(u)=0\}$ and $G_i=G_{i-1}-C_i$. Let $k$ be the number of vertices in the longest directed path in $G$. Clearly, $V(G_i)=\emptyset$ for every $i\geq k$. For each $u\in V(G)$, we define $\beta(u)$ to be the unique integer $i\in\{1,2,\ldots,k\}$ such that $u\in C_i$. 
It well-known (and easy to verify)  that $\beta$ is a proper vertex coloring of $\hat G$ using $k$ colors, and  whenever there is
an edge between $u$ and $v$, with $\beta(u) > \beta(v)$, it is oriented in $G$ from $u$ to $v$, i.e. $G$ is indeed the natural orientation of $\hat G$ with respect to $\beta$.

Let $\alpha=\refinement(\beta,<)$ (here, $<$ is the usual ordering on $\mathbb{N}^{>0}$). Let $D$ be the oriented graph having $V(D)=V(G)$ and $E(D)=\{(u,v)\in E(G)\colon\alpha(u)>\alpha(v)\}$. By Proposition~\ref{prop1}, it follows that $\alpha$ is an out-tree coloring of $D$. Since $\alpha$ is a proper vertex coloring of $G$, we know that $\alpha$ uses at least $\chi(G)$ colors. Now applying Lemma~\ref{lem:generallemma} on $G$, $D$, and $T$, we get that if one of the conditions in the statement of the lemma is satisfied, then the out-tree $T$ is present as an induced subgraph of $G$.

For the case when $T$ is an in-tree, we simply consider the oriented graph $G'$ that is obtained from $G$ by reversing the direction of each edge. Clearly, $\chi(G')=\chi(G)$, and $G'$ is $\mathcal{B}^+_r$-free if and only if $G$ is $\mathcal{B}^-_r$-free. We then have by the argument above that $G'$ contains all out-trees on $s$ vertices as induced subgraphs. Since every in-tree on $s$ vertices can be obtained by reversing the directions of the edges of some out-tree on $s$ vertices, we now have that every in-tree on $s$ vertices is an induced subgraph of $G$.
\end{proof}

\subsection{Oriented trees in bikernel-perfect oriented graphs of girth at least 5}
\newcommand{\goodcoloring}{parity-coloring}
\begin{definition}
A proper coloring $\alpha$ of an oriented graph using colors from $\mathbb{N}^{>0}$ is said to be a \emph{\goodcoloring} if for each vertex $u\in V(G)$, $\alpha(N^+(u))\supseteq\{i\in\{1,2,\ldots,\alpha(u)-1\}\colon i\equiv 0\mod 2\}$ and $\alpha(N^-(u))\supseteq\{i\in \{1,2,\ldots,\alpha(u)-1\}\colon i\equiv 1\mod 2\}$.
\end{definition}
We now prove a lemma using techniques similar to the ones used in Lemma~\ref{lem:generallemma}. Recall that the girth of an oriented graph is simply the girth of its underlying undirected graph (see Section~\ref{sec:introduction}).

\begin{lemma}\label{lem:orientedtree}
Let $G$ be an oriented graph having girth $g\geq 5$. Suppose that $G$ has a \goodcoloring\ using colors $\{1,2,\ldots,k\}$ in which the color $k$ has been given to some vertex. If $k\geq 2s^{1+4/(g-4)}$, then $G$ contains every oriented tree on $s$ vertices as an induced subgraph.
\end{lemma}
\begin{proof}
For every vertex $u\in V(G)$ and every even integer $i\in\{1,2,\ldots,\alpha(u)-1\}$, let $u^i$ denote an arbitrarily chosen vertex in $N^+(u)$ having color $i$. Similarly, for every vertex $u\in V(G)$ and every odd integer $i\in \{1,2,\ldots,\alpha(u)-1\}$, let $u^i$ denote an arbitrarily chosen vertex in $N^-(u)$ having color $i$.
Let $T$ be any oriented tree on $s$ vertices and let $w_1,w_2,\ldots,w_s$ be an ordering of $V(T)$ obtained by a depth-first traversal of $\hat{T}$ (the underlying undirected tree of $T$) starting from a vertex $w_1$ that we fix as the root. For each $i\in\{2,3,\ldots,s\}$, let $p(i)$ be the integer such that $w_{p(i)}$ is the parent of $w_i$ in the rooted tree $\hat{T}$. As $w_1$ corresponds to the root of $\hat{T}$, the parent of every vertex appears before it in the ordering $w_1,w_2,\ldots,w_s$, and so we have $p(i)<i$.

For any $i\in\{1,2,\ldots,s\}$, we denote by $T_i$ the oriented tree $T[\{w_1,w_2,\ldots,w_i\}]$. Let $T'_i$ be an induced subgraph of $G$ that is isomorphic to $T_i$ and having vertices $v_1,v_2,\ldots,v_i$ corresponding to $w_1,w_2,\ldots,w_i$. Clearly, for $j\in\{2,3,\ldots,i\}$, $v_{p(j)}$ is the parent of $v_j$ in the rooted tree $\hat{T'_i}$ having $v_1$ as its root. We say that $T'_i$ is a ``good'' tree if $\alpha(v_1)=k$, and for $2\leq j\leq i$:
\begin{enumerate}
\item $\alpha(v_{j-1})>\alpha(v_j)$, and
\item for $\alpha(v_{j-1})>t>\alpha(v_j)$, $G[\{v_1,v_2,\ldots,v_{j-1},v_{p(j)}^t\}]$ is not isomorphic to $T_j$.
\end{enumerate}
Let $i$ be the maximum integer such that there is a good tree $T'_i$ isomorphic to $T_i$ in $G$ (since there is a vertex having color $k$, we have that $i$ exists). If $i\geq s$, then we are done. So let us assume for the sake of contradiction that $i\leq s-1$. For $2\leq j\leq i$, let $U_j=\{t\colon\alpha(v_j)<t<\alpha(v_{j-1})\}$.
As $T'_i$ is a good tree in $G$, if $v_j\in N^+(v_{p(j)})$ (resp. $v_j\in N^-(v_{p(j)})$) then for every even (resp. odd) $t\in U_j$,
we have $N(v_{p(j)}^t)\cap(\{v_1,v_2,\ldots,v_{j-1}\}\setminus\{v_{p(j)}\})\neq\emptyset$.
Let $U'=\{t\colon 1\leq t<\alpha(v_i)\}$.
If $w_{i+1}\in N^+_T(w_{p(i+1)})$ (resp. $w_{i+1}\in N^-_T(w_{p(i+1)})$) and for some even (resp. odd) $t\in U'$, we have $N(v_{p(i+1)}^t)\cap(\{v_1,v_2,\ldots,v_i\}\setminus\{v_{p(i+1)}\})=\emptyset$, then taking $t$ to be the largest such integer in $U'$, we have that $G[\{v_1,v_2,\ldots,v_i\}\cup\{v_{p(i+1)}^t\}]$ is a good tree isomorphic to $T_{i+1}$, which contradicts our choice of $i$. We can therefore conclude that if $w_{i+1}\in N^+_T(w_{p(i+1)})$ (resp. $w_{i+1}\in N^-_T(w_{p(i+1)})$), then for every
even (resp. odd) $t\in U'$, we have $N(v_{p(i+1)}^t)\cap(\{v_1,v_2,\ldots,v_i\}\setminus\{v_{p(i+1)}\})\neq\emptyset$.
Let $U=U'\cup\bigcup_{j=2}^i U_i$. Clearly, $|U|+i=k$.

It may be helpful at this point to recall that a ``path'' or ``cycle'' in an oriented graph is simply a path or cycle, respectively, in its underlying undirected graph (see Section~\ref{sec:introduction}).

First, let us show that $i\geq g$. Assume for the sake of contradiction that $i<g$.  We claim that $|U_j|\leq 1$ for $2\leq j\leq i$. Suppose that there exists $j\in\{2,3,\ldots,i\}$ such that $|U_j|\geq 2$. Then if $v_j\in N^+(v_{p(j)})$ (resp. $v_j\in N^-(v_{p(j)})$), there is an even (resp. odd) $t\in U_j$ such that there exists $z\in N(v_{p(j)}^t)\cap (\{v_1,v_2,\ldots,v_{j-1}\}\setminus\{v_{p(j)}\})$. But then the path in $T'_i$ between $z$ and $v_{p(j)}$ together with $v_{p(j)}^t$ forms a cycle of length at most $j\leq i<g$ in $G$, which is a contradiction. So we can assume that $|U_j|\leq 1$ for $2\leq j\leq i$. Now suppose that $w_{i+1}\in N^+_T(w_{p(i+1)})$ (resp. $w_{i+1}\in N^-_T(w_{p(i+1)})$ and there is an even (resp. odd) $t\in U'$. Then there exists $z\in N(v_{p(i+1)}^t)\cap (\{v_1,v_2,\ldots,v_i\}\setminus\{v_{p(i+1)}\})$. Note that the path in $T'_i$ between $z$ and $v_{p(i+1)}$ together with $z$ forms a cycle of length at most $i+1$. Since we assumed that $i<g$, we also have that this cycle has length at least $i+1$, which implies that this cycle has length exactly $i+1$. This means that the path between $z$ and $v_{p(i+1)}$ contains all the vertices of $T'_i$, or to be more precise, $z=v_1,v_2,\ldots,v_i=v_{p(i+1)}$ is a path. Thus, if $w_{i+1}\in N^+_T(w_{p(i+1)})$ (resp. $w_{i+1}\in N^-_T(w_{p(i+1)})$ and there exist distinct even (resp. odd) $t,t'\in U'$, we have $v_1\in N(v_{p(i+1)}^t)\cap N(v_{p(i+1)}^{t'})$, which implies that there is a cycle $v_{p(i+1)},v_{p(i+1)}^t,v_1,v_{p(i+1)}^{t'},v_{p(i+1)}$ of length 4 in $G$, which is a contradiction. It follows that $|U'|\leq 3$. Then $|U|\leq 2+i$, which gives $i\geq k-2-i$, or in other words, $i\geq (k-2)/2$. From the statement of the lemma, we have $k\geq 2s+1$, as $k$ and $s$ are integers, implying that \(i\geq (2s-1)/2>s-1\), contradicting the assumption that $i\leq s-1$. We can therefore assume that $i\geq g$.

We shall construct an auxiliary undirected graph \(H\) with \(V(H) = V(T'_i)\) and \(E(H) = E(\hat{T'_i}) \cup E'\), where \(E'\) is a set of edges defined as follows. For $j\in\{2,3,\ldots,i\}$ such that $v_j\in N^+(v_{p(j)})$ (resp. $v_j\in N^-(v_{p(j)})$) and even (resp. odd) \(t \in U_j\), the edge \(v_{p(j)}x\) is added to \(E'\), where $x$ is an arbitrarily chosen vertex in $N(v_{p(j)}^t)\cap (\{v_1,v_2,\ldots,v_{j-1}\}\setminus\{v_{p(j)}\})$. Note that since $v_{p(j)}$ is a neighbour of $v_{p(j)}^t$ in $G$ and $g(G)=g\geq 5$, $v_{p(j)}x$ is not an edge of $\hat{G}$. Also, if any other neighbour of $v_{p(j)}$ was adjacent to $x$, there would be a cycle of length four in $G$, again contradicting the fact that $G$ has girth at least 5. Thus, no edge of $E(\hat{G})$ is present in $E'$ and no edge is added twice to $E'$. Similarly, if $w_{i+1}\in N^+(w_{p(i+1)})$ (resp. $w_{i+1}\in N^-(w_{p(i+1)})$), then for each even (resp. odd) $t\in U'$, the edge $v_{p(i+1)}x$ is added to $E'$, where $x$ is an arbitrarily chosen vertex in $N(v_{p(i+1)}^t)\cap (\{v_1,v_2,\ldots,v_i\}\setminus\{v_{p(i+1)}\})$. Thus, $H$ is a simple graph and \(|E'| \geq \left\lfloor\frac{|U'|}{2}\right\rfloor+\sum_{j=2}^i \left\lfloor\frac{|U_j|}{2}\right\rfloor\geq\frac{|U|-i}{2}=\frac{k}{2}-i\). This gives $k\leq 2|E'|+2i$.

Now, we claim that \(g(H) \geq g/2\). Given any cycle \(C\) of \(H\), we can replace each edge of \(C\) by at most 2 edges of \(G\) to obtain a closed walk in $G$ containing at most $2g(H)$ edges (the closed walk will actually be a cycle). Thus, if \(g(H) < g/2\), then there is a cycle in \(G\) of length less than \(g\), which is a contradiction.

\medskip
\noindent From Lemma~\hyperref[ext1]{\ref{lem:extremal}\ref{ext1}}, if \(g(H)\) is odd, we get: 
\[|E(H)| \  < \  \frac{i^{1+2/(g(H)-1)} + i}{2} \  \leq \  \frac{i^{1+4/(g-2)} + i}{2} \ < \ \frac{i^{1+4/(g-4)} + i}{2}\]

\medskip
\noindent From Lemma~\hyperref[ext2]{\ref{lem:extremal}\ref{ext2}}, if \(g(H)\) is even, we get: 
\[|E(H)| \  < \  \bigg(\frac{i}{2}\bigg)^{1+2/(g(H)-2)} + \frac{i}{2} \  \leq \  \bigg(\frac{i}{2}\bigg)^{1+4/(g-4)} + \frac{i}{2}  \ < \ \frac{i^{1+4/(g-4)} + i}{2}\]

\medskip
\noindent Combining both the cases, we get 
\[|E(H)| \ < \ \frac{i^{1+4/(g-4)} + i}{2}\]

\noindent Since \(|E(H)| = |E'|+(i-1)\)
\[|E'| < \frac{i^{1+4/(g-4)} - i + 2}{2}\]
\noindent Since \(k\leq 2|E'|+2i\), we get:
\begin{eqnarray*}
k \: &<& \: i + i^{1+4/(g-4)} + 2= i(1 + 2/i+i^{4/(g-4)})\\
&\leq& i(e^{2/i}+i^{4/(g-4)})\qquad\mbox{(since }1+2/i\leq e^{2/i}\mbox{)}\\
&\leq& i(e^{2/(g-4)}+i^{4/(g-4)})\qquad\mbox{(since $i\geq g\geq g-4$)}\\
&=&i(\sqrt{e}^{4/(g-4)}+i^{4/(g-4)})
\end{eqnarray*}

By our assumption that $i\geq g$, we get that \(i \geq 2 > \sqrt{e}\). Using this in the above inequality, we get:
\[k < 2i^{1+4/(g-4)}\]

Now by using the bound on $k$ given in the statement of the lemma, we get:
$$ 2s^{1+4/(g-4)} \leq k < 2i^{1+4/(g-4)}$$
This implies that $i\geq s$, which is a contradiction.
\end{proof}

\bikernelthm*
\begin{proof}
Let $G_0=G$. For each odd (resp. even) $i\geq 1$, let $C_i$ be an antikernel (resp. kernel) of $G_{i-1}$ and let $G_i=G_{i-1}-C_i$. Since $G$ is bikernel-perfect, there exists a smallest integer $k$ for which $V(G_k)=\emptyset$. Note that $\{C_1,C_2,\ldots,C_k\}$ is a partition of $V(G)$. For each vertex $u\in V(G)$, define $\beta(u)$ to be the unique integer $i$ in $\{1,2,\ldots,k\}$ such that $u\in C_i$. Since for $1\leq i<j\leq k$, every vertex in $C_j\subseteq V(G_i)$ has an in-neighbour (resp. out-neighbour) in $C_i$ when $i$ is odd (resp. even), we can conclude that $\beta$ is a \goodcoloring\ of $G$ using the colors in $\{1,2,\ldots,k\}$. Since $k\geq\chi(G)\geq 2s^{1+\frac{4}{g-4}}$, it follows from Lemma~\ref{lem:orientedtree} that $G$ contains every oriented tree on $s$ vertices as an induced subgraph.
\end{proof}
\subsection{Oriented trees in $\mathcal{B}_r$-free oriented graphs}

In this section, we will show a much stronger result than what is required to prove Theorem~\ref{thm:outtree}\ref{outtbr}, albeit for a slightly higher bound on the chromatic number. We show that if the chromatic number of an oriented $\mathcal{B}_r$-free graph is at least $(r-1)(s-1)(s-2)+2s+1$, then it contains every oriented tree on at most $s$ vertices as induced subgraph. This generalizes Theorem~\ref{thm:outtree}\ref{outtbr} from out-trees in directed acyclic graphs to arbitrarily oriented trees in general oriented graphs. Recall that if $A$ is an oriented graph and $\mathcal{F}$ a family of oriented graphs, then $\overline{f}(A,\mathcal{F})$ is the smallest positive integer
such that every $\mathcal{F}$-free oriented graph having chromatic number at least $\overline{f}(A,\mathcal{F})$ contains $A$ as an induced subgraph. Also recall that $\overline{f}(s,\mathcal{F})$ is the smallest positive integer such that every $\mathcal{F}$-free oriented graph having chromatic number at least $\overline{f}(s,\mathcal{F})$ contains every oriented tree on $s$ vertices as an induced subgraph. Thus, in this section, we prove that $\overline{f}(s,\mathcal{B}_r)\leq (r-1)(s-1)(s-2)+2s+1$.

For $s \ge 1$, let $q(s)$ be the smallest positive integer such that every $q(s)$-chromatic directed graph contains every oriented tree on $s$ vertices as a (not necessarily induced) subgraph. Burr~\cite{Burr} (Theorem~\ref{thm:burr}) proved that $q(s) \le (s-1)^2$, and conjectured that $q(s) = 2s-2$. Addario-Berry et al.~\cite{Addario} obtained the following improved bound for $q(s)$.

\begin{theorem} [Addario-Berry et al.~\cite{Addario}] \label{thm:addario}
$q(s) \le s^2/2 - s/2 + 1$.
\end{theorem}

We need an analogous version of this result, but for induced subgraphs. But as noted in Section~\ref{sec:ourresults}, we cannot hope to have such a result even for triangle-free graphs, as $\overline{f}(T,C_3)$ does not exist even for the case when $T$ corresponds to certain orientations of a 4-path. We thus restrict ourselves to $\mathcal{B}_r$-free graphs. 
In their paper, Addario-Berry et al. describe  several approaches to prove Theorem~\ref{thm:addario}, with slight changes in the upper bounds for $q(s)$. We show that their most basic approach can be
adapted to get an induced version of the theorem for $\mathcal{B}_r$-free graphs; the other methods in~\cite{Addario} seem difficult to adapt for the induced version. 
\medskip

An undirected graph $G$ is said to be $k$-degenerate if every subgraph of $G$ has minimum degree at most $k$. It is a well-known and easy to see fact that the chromatic number of a $k$-degenerate graph is at most $k+1$. The following observation will be useful for us.
\begin{lemma}\label{lem:avgdegree}
Let $k>2$. If $G$ is a graph whose every subgraph has average degree at most $k$, then either $G$ contains a clique on $k+1$ vertices or $\chi(G)\leq k$.
\end{lemma}
\begin{proof}
Suppose that $\chi(G)>k$. Then $G$ is not $(k-1)$-degenerate. From $G$, we repeatedly remove vertices of degree at most $k-1$ until we get a graph $G'$ whose minimum degree is $k$.
Notice that we will be done if we prove that $\chi(G')\leq k$ or $G'$ contains a clique on $k+1$ vertices.
From the fact that every subgraph of $G$ has average degree at most $k$, we know that the average degree of $G'$ is at most $k$. Since the minimum degree of $G'$ is $k$, it follows that the maximum degree of $G'$ is also $k$, or in other words, $G'$ is a $k$-regular graph.
This means that if $G'$ is a complete graph, then it is a complete graph on $k+1$ vertices, in which case we are done. As $k>2$, $G'$ is not an odd cycle, and therefore by Brooks' theorem, we have $\chi(G')\leq k$.
\end{proof}

A vertex $v$ is said to be an \emph{out-leaf} of an oriented tree $T$ if $d^-_T(v)=1$ and $d^+_T(v)=0$. Similarly, a vertex $v$ is said to be an \emph{in-leaf} of an oriented tree $T$ if $d^+_T(v)=1$ and $d^-_T(v)=0$. A vertex is said to be a \emph{leaf} of an oriented tree $T$ if it is either an out-leaf or an in-leaf of $T$.
It is easy to see that for $r\geq 2$, $\overline{f}(T,\mathcal{B}_r)=1$ when $|V(T)|=1$. The following lemma will help us in bounding $\overline{f}(T,\mathcal{B}_r)$ for arbitrary oriented trees $T$.

\newcommand{\val}{(r-1)(s-2)+t}

\begin {lemma} \label{lem:oriented-gyarfas}
Let $T$ be an oriented tree on $s$ vertices, where $s\geq 2$, and let $L$ be either the set of all out-leaves of $T$ or the set of all in-leaves of $T$. Let $t=|L|$ and $T'=T-L$.
Then for $r\geq 2$, $\overline{f}(T,\mathcal{B}_r)\leq \overline{f}(T',\mathcal{B}_r) + 2(\val)$.
\end {lemma}

\begin{proof}
If $s=2$, then any oriented graph $D$ having $\chi(D)\geq 2$ contains an induced subgraph isomorphic to $T$, implying that $\overline{f}(T,\mathcal{B}_r)\leq 2$, and so the bound given in the statement of the lemma holds (since in this case, we have $t=1$). Therefore, from here onward, we assume that $s\geq 3$ (which implies that $t\geq 1$).
	
We assume that $L$ is a set of out-leaves of $T$; the case when $L$ is a set of in-leaves is symmetric. Let $D$ be any oriented graph with $\chi(D) \ge \overline{f}(T',\mathcal{B}_r)+2(\val)$. Define $S= \{ x \in V(D) : d^+_D(x) \le \val\}$.

We claim that $\chi(D[S]) \le 2(\val)$.
Indeed for any $S'\subseteq S$, each vertex in the graph $D[S']$ has out degree at most $\val$ and therefore the number of edges in $D[S']$ is at most $(\val)|S'|$. It follows that the average degree of $\hat{D}[S']$ is at most $2(\val)$. 
Thus every subgraph of $\hat{D}[S]$ has average degree at most $2(\val)$.
It is easy to see that $2(\val)>2$. Then by Lemma~\ref{lem:avgdegree}, we have that either $\chi(D[S])\leq 2(\val)$ or $D[S]$ contains a tournament on $2(\val)+1\geq 2r+1$ vertices. It follows that there exists a vertex having out-degree at least $r$ in this tournament, which implies that the tournament contains a subgraph isomorphic to a graph in $\mathcal{B}_r$. Since this is a contradiction, we can conclude that $\chi(D[S])\leq 2(\val)$.

From this, we can infer that $\chi(D') \ge  \overline{f}(T', \mathcal{B}_r)$, where $D' = D-S$. Then by the definition of $\overline{f}(T', \mathcal{B}_r)$, we have that $D'$ contains an induced subgraph that is isomorphic to $T'$.
Let $L=\{v_1,v_2,\ldots,v_t\}$.
For each $j\in\{1,2,\ldots,t\}$, we denote by $T_j$ the oriented tree $T[V(T')\cup\{v_1,v_2,\ldots,v_j\}]$. We let $T_0=T'$.
Let $i$ be the maximum integer such that there is an 
induced subgraph of $D$ that is isomorphic to $T_i$ having the property that the vertices in it corresponding to the vertices in $V(T_i-L)=V(T')$ are all from $V(D')$.
Since there is an induced subgraph of $D'$ that is isomorphic to $T'=T_0$, we know that $i$ exists. We shall prove that $i=t$, which will complete the proof. 

Suppose that $i<t$.
By the definition of $i$, we know that there is an induced subgraph $F$ of $D$ that is isomorphic to $T_i$ in which the vertices corresponding to $V(T')$ are all from $V(D')$.
Clearly, $|V(F)|<|V(T)|=s$.
Let $u$ be the only neighbour of $v_{i+1}$ in $T$ and let $u'$ be the vertex corresponding to $u$ in $F$. Clearly, $u\in V(T')$, and therefore $u'\in V(D')$.
This implies that $|N^+_D(u')|>\val$.
For each vertex $x\in V(F)\setminus\{u'\}$, define $W_x=N[x]\cap N^+(u')$.
For each $x\in V(F)\setminus N^+_{F}[u']$, $|W_x|\leq r-1$, as otherwise $r$ vertices from $W_x$ together with $x$ and $u'$ will form a subgraph of $D$ that is isomorphic to a graph in $\mathcal{B}_r$ (note that since $F$ is an induced subgraph of $D$, we have $x\notin W_x$ in this case).
Similarly, for each $x\in N^+_{F}(u')$, $|W_x|\leq r$, as otherwise $r$ vertices from $W_x\setminus\{x\}$ together with $x$ and $u'$ will form a subgraph of $D$ that is isomorphic to a graph in $\mathcal{B}_r$.
Let $Z\subseteq N^+_{F}(u')$ be the leaves of $F$ in $N^+_{F}(u')$. Since the vertices in $Z$ are all out-leaves of $F$, and $F$ has at most as many out-leaves as $T$, it follows that $|Z|\leq t$.
For every vertex $x\in N^+_{F}(u')\setminus Z$, let $y(x)$ be an arbitrarily chosen vertex in $N_{F}(x)\setminus\{u'\}$. Let $Y=\{y(x)\colon x\in N^+_{F}(u')\setminus Z\}$. Clearly, as $F$ is a tree that is induced in $D$, we have that for each $x\in N^+_{F}(u')\setminus Z$, $y(x)\notin N[u']$ and for any $x'\in N^+_{F}(u')\setminus Z$ that is distinct from $x$, we have $y(x)\neq y(x')$. Thus, we have $Y\cap N^+_{F}(u')=\emptyset$ and $|Y|=|N^+_{F}(u')\setminus Z|$. For each $x\in Y$, $|W_x\setminus\{y^{-1}(x)\}|\leq r-2$, since otherwise $r-1$ vertices from $W_x\setminus\{y^{-1}(x)\}$, $x$, $y^{-1}(x)$ and $u'$ will form a $\mathcal{B}_r$ in $D$.

The number of vertices in $N^+_D(u')$ that are either in $V(F)$ or are adjacent to a vertex in $V(F)$ other than $u'$ is at most
\begin{eqnarray*}
\left|\bigcup_{x\in V(F)\setminus\{u'\}} W_x\right|&\leq&\sum_{x\in N^+_{F}(u')} |W_x|+ \sum_{x\in Y} |W_x\setminus\{y^{-1}(x)\}|+\sum_{x\in V(F)\setminus (N^+_{F}[u']\cup Y)} |W_x|
\end{eqnarray*}
The second term comes from the fact that for every $x\in N^+_F(u')$, we have $x\in W_x$ and $x\in W_{y(x)}$. Substituting the upper bounds derived above, we get
\begin{eqnarray*}
\left|\bigcup_{x\in V(F)\setminus\{u'\}} W_x\right|&\leq&\sum_{x\in N^+_{F}(u')} r + \sum_{x\in Y} (r-2) + \sum_{x\in V(F)\setminus (N^+_{F}[u']\cup Y)} (r-1)\\
&=& \sum_{x\in N^+_{F}(u')\setminus Z} r+ \sum_{x\in Z} r + \sum_{x\in Y} (r-1) - |Y|+ \sum_{x\in V(F)\setminus (N^+_{F}[u']\cup Y)} (r-1)\\
\end{eqnarray*}
Since $|Y|=|N^+_F(u')|\setminus Z$, we can take the term $-|Y|$ inside the first term of the right hand side, which gives us
\begin{eqnarray*}
\left|\bigcup_{x\in V(F)\setminus\{u'\}} W_x\right|&\leq& \sum_{x\in N^+_{F}(u')\setminus Z} (r-1)+ \sum_{x\in Z} r + \sum_{x\in Y} (r-1) + \sum_{x\in V(F)\setminus (N^+_{F}[u']\cup Y)} (r-1)\\
&=&\sum_{x\in N^+_{F}(u')\setminus Z} (r-1) + |Z| + \sum_{x\in Z} (r-1) + \sum_{x\in Y} (r-1) + \sum_{x\in V(F)\setminus (N^+_{F}[u']\cup Y)} (r-1)\\
&=&|Z| + \sum_{x\in V(F)\setminus\{u'\}} (r-1)\\
&\leq&t+(r-1)(s-2)
\end{eqnarray*}

Since $|N^+_D(u')|>(r-1)(s-2)+t$, there exists a vertex $v'$ in $N^+_D(u')$ that is neither a vertex of $F$ nor adjacent to any vertex of $F$ other than $u'$. It can be seen that $V(F)\cup\{v'\}$ induces a subgraph of $D$ that is isomorphic to $T_{i+1}$ in which the vertices corresponding to $V(T')$ are all from $D'$. This contradicts our choice of $i$.
\end{proof}

As defined in~\cite{Addario}, for an oriented tree $T$, the parameter $\st(T)$ is defined as follows. Let $L_1$ be the set of out-leaves of $T$ and $L_2$ the set of in-leaves of $T$. Then $\st(T)=0$ if $|V(T)|=1$, and otherwise $\st(T)=1+\min\{\st(T-L_1),\st(T-L_2)\}$.

\begin{theorem}\label{thm:oriented_gyarfas}
For any oriented tree $T$ having $s$ vertices and $r\geq 2$,
$\overline{f}(T,\mathcal{B}_r) \le  (r-1)(2s-\st(T)-3)\st(T)+2s+1$.
\end{theorem}
\begin{proof}  
We prove this by induction on $\st(T)$. If $\st(T)=0$, then $|V(T)|=1$, implying that $\overline{f}(T,\mathcal{B}_r)=1$, and therefore the upper bound given by the theorem holds.
So we shall assume that $\st(T)\geq 1$. Let $T'$ be the tree obtained by removing all the out-leaves or all the in-leaves of $T$ so that $\st(T')=\st(T)-1$.
Let $t$ denote the number of out-leaves or in-leaves that were removed from $T$ to obtain $T'$. Then $|V(T')|=s-t$.
By the induction hypothesis, we have $\overline{f}(T',\mathcal{B}_r)\leq (r-1) (2|V(T')| - \st(T') -3) \st(T') + 2|V(T')| + 1 =(r-1)(2(s-t)-\st(T)-2)(\st(T)-1) + 2(s-t) + 1$. By Lemma~\ref {lem:oriented-gyarfas}, we now have
\begin{eqnarray*}
\overline{f}(T,\mathcal{B}_r)&\leq&\overline{f}(T',\mathcal{B}_r)+2((r-1)(s-2)+t)\\
&\leq&(r-1) (2s-2t-\st(T)-2)(\st(T)-1)+  2(s-t) +1 +2((r-1)(s-2)+t)\\
&=&(r-1)(2s-2t-\st(T)-2)\st(T)-(r-1)(2s-2t-\st(T)-2)+2s+1-2t +2(r-1)(s-2) +2t\\
&=&(r-1)(2s-\st(T)-2)\st(T)-(r-1)(2s+2\st(T)t-2t-\st(T)-2)+2s+1+2(r-1)(s-2)\\
&=&(r-1)(2s-\st(T)-2)\st(T)-(r-1)(2s+2\st(T)t-2t-\st(T)-2-2s+4)+2s+1\\
&=&(r-1)(2s-\st(T)-2)\st(T)-(r-1)(2\st(T)t-2t-\st(T)+2)+2s+1\\
&=&(r-1)(2s-\st(T)-2)\st(T)-(r-1)(\st(T)(2t-1)-2t+2)+2s+1\\
&=&(r-1)(2s-\st(T)-2-2t+1)\st(T)+(r-1)(2t-2)+2s+1\\
&=&(r-1)(2s-\st(T)-2t-1)\st(T)+(r-1)(2t-2)+2s+1\\
&\leq&(r-1)(2s-\st(T)-2t-1)\st(T)+(r-1)(2t-2)\st(T)+2s+1\\
&=&(r-1)(2s-\st(T)-2t-1+2t-2)\st(T)+2s+1\\
&=&(r-1)(2s-\st(T)-3)\st(T)+2s+1
\end{eqnarray*}
\end{proof}
\medskip

Note that the upper bound given in Theorem~\ref{thm:oriented_gyarfas} is maximized for a tree $T$ such that $\st(T)\in\{|V(T)|-1,|V(T)|-2\}$. We thus get the following result.

\orientedgstthm*

\bibliographystyle{plain}
\bibliography{References}

\end{document}